\documentclass[a4paper,twoside]{article}

\usepackage{comment}
\usepackage{amsmath, amsthm, amssymb,mathtools}
\usepackage{graphicx}
\usepackage{tikz}
\usepackage{paralist} 
\usepackage{fullpage}
\usepackage{ytableau}
\usepackage{caption}
\usepackage[labelformat=simple]{subcaption}

\usepackage{booktabs}

\usepackage[outline]{contour}
\contourlength{2pt}

\newtheorem{theorem}{Theorem}[section]
\newtheorem{lemma}[theorem]{Lemma}

\theoremstyle{definition}
\newtheorem{definition}[theorem]{Definition}

\numberwithin{equation}{section}

\DeclareMathOperator{\asym}{\mathbf{ASym}}
\DeclareMathOperator{\sym}{\mathbf{Sym}}
\DeclareMathOperator{\sgn}{sgn}
\DeclareMathOperator{\ct}{CT}
\DeclareMathOperator{\E}{E}
\DeclareMathOperator{\fd}{\Delta}
\DeclareMathOperator{\bd}{\delta}
\DeclareMathOperator{\id}{id}
\DeclareMathOperator{\Qfd}{{}^\textit{M}\Delta}
\DeclareMathOperator{\Qbd}{{}^\textit{M}\delta}
\DeclareMathOperator{\subsets}{\mathbf{Subsets}}

\DeclareMathOperator{\AST}{ASTZ}
\DeclareMathOperator{\CSSPP}{CSSPP}
\DeclareMathOperator{\MT}{MT}

\newmuskip\pFqmuskip

\newcommand*\pFq[6][8]{%
	\begingroup 
	\pFqmuskip=#1mu\relax
	\mathchardef\normalcomma=\mathcode`,
	\mathcode`\,=\string"8000
	\begingroup\lccode`\~=`\,
	\lowercase{\endgroup\let~}\pFqcomma
	{}_{#2}F_{#3}{\left[\left.\genfrac..{0pt}{}{#4}{#5}\right|#6\right]}%
	\endgroup
}
\newcommand{\pFqcomma}{{\normalcomma}\mskip\pFqmuskip}

\newcounter{xmt}
\newcounter{ymt}

\newcommand\mt[1]{
	\setcounter{ymt}{-1}
	\foreach \p in {#1} {
		\addtocounter{ymt}{1}
		\setcounter{xmt}{\value{ymt}-1}
		\foreach \q in \p {
			\addtocounter{xmt}{2}
			\node at (\value{xmt}*.5,\value{ymt}*.5) {\q};   
		}
	}
}

\usepackage[bookmarks,hyperindex,colorlinks,pdftex]{hyperref}

\usepackage{doi}

\begin{document}

\title{A Fourfold Refined Enumeration of Alternating Sign Trapezoids}

\author{Hans H{\"o}ngesberg\thanks{The author acknowledges support from the Austrian Science Foundation FWF, SFB grant F50.}
	}
	
\date{}


\maketitle

\begin{abstract}
	Alternating sign trapezoids have recently been introduced as a generalisation of alternating sign triangles. Fischer established a threefold refined enumeration of alternating sign trapezoids and provided three statistics on column strict shifted plane partitions with the same joint distribution. In this paper, we are able to add a new pair of statistics to these results. More precisely, we consider the number of $-1$s on alternating sign trapezoids and introduce a corresponding statistic on column strict shifted plane partitions that has the same distribution. More generally, we show that the joint distributions of the two quadruples of statistics on alternating sign trapezoids and column strict shifted plane partitions, respectively, coincide. In addition, we provide a closed-form expression for the $2$-enumeration of alternating sign trapezoids.
\end{abstract}


\section{Introduction}

Since their introduction in the early 1980s, alternating sign matrices have generated great interest among combinatorialists. Mills, Robbins and Rumsey \cite{MRR83} conjectured them to be equinumerous with descending plane partitions, which had been enumerated by Andrews \cite{And79} a few years earlier; this was finally proved over a decade later first by Zeilberger \cite{Zei96a} and shortly thereafter by Kuperberg \cite{Kup96}. Since then, the spellbinding research of alternating sign matrices has revealed new equinumerous classes of combinatorial objects but finding bijections remains one of the most challenging problems. Equally distributed statistics on these objects might finally lead to those eagerly awaited bijections. Embracing this idea, we provide a fourfold refined enumeration of alternating sign trapezoids, a recently defined  generalisation of alternating sign triangles. Moreover, we establish four statistics on certain column strict shifted plane partitions with the same joint distribution. Thus, we generalise the recent refined enumerations of alternating sign trapezoids and of column strict shifted plane partitions by Fischer \cite{Fis}. 

We start by introducing alternating sign trapezoids and column strict shifted plane partitions together with four statistics on each of these classes of objects.

\begin{definition}\label{def:ASTZs}
	For given integers~$n \geq 1$ and $l \geq 2$, an \emph{$(n,l)$-alternating sign trapezoid} is an array of $-1$s, $0$s and $+1$s in a trapezoidal shape with $n$ rows
	of the following form
	\begin{equation*}
	\begin{array}[t]{ccccccccc}
	a_{1,1}&a_{1,2}&\cdots&\cdots&\cdots&\cdots&\cdots&\cdots&a_{1,2n+l-2}\\
	&a_{2,2}&\cdots&\cdots&\cdots&\cdots&\cdots&a_{2,2n+l-3}&\\
	&&\ddots&&&&\reflectbox{$\ddots$}&&\\
	&&&a_{n,n}&\cdots&a_{n,n+l-1}&&&
	\end{array}
	\end{equation*}
	such that the following four conditions hold: the nonzero entries alternate in sign in each row and each column; the topmost nonzero entry in each column is $1$ (if existent); the entries in each row sum to $1$; and the entries in the central $l-2$ columns sum to $0$.

	An \emph{$(n,1)$-alternating sign trapezoid} is defined as above with the exception that the bottom row, which consists of a single entry in this case, does not have to add up to $1$, but where this entry can either be $1$ or $0$.
\end{definition}

Note that the notion of $(n,1)$-alternating sign trapezoids coincide with the notion of \emph{quasi alternating sign triangles of order $n$}. Furthermore, $(n,3)$-alternating sign trapezoids are in bijective correspondence with \emph{alternating sign triangles of order $n+1$}. Both alternating sign triangles and quasi alternating sign triangles were introduced by Ayyer, Behrend and Fischer \cite{ABF}.

From the definition, it follows that the entries in each column of an alternating sign trapezoid sum to $0$ or $1$. A column whose entries sum to $1$ is called a \emph{$1$-column}. If, in addition, the bottom entry of a $1$-column is $0$, we call the column a \emph{$10$-column}.
Note that the number of $1$-columns in any $(n,l)$-alternating sign trapezoid is exactly $n$ if $l \neq 1$ because the sum of all entries in an $(n,l)$-ASTZ is $n$ as all rows add up to $1$; otherwise, it is $n$ or $n-1$.
\begin{figure}[ht]
	\centering
	\begin{equation*}
	\begin{array}[t]{cccccccccc}
	0 & 0 & 0 & 0 & 0 & 1 & 0 & 0 & 0 & 0\\
	& 1 & 0 & 0 & 0 & -1 & 0 & 1 & 0 &  \\
	&   & 0 & 0 & 0 & 0 & 1 & 0 &   &  \\
	&   &   & 1 & 0 & 0 & 0 &   &   &  \\
	\end{array}
	\end{equation*}
	\caption{$(4,4)$-alternating sign trapezoid $A$ with $\mu(A)=1$, $r(A)=2$, $p(A)=0$ and $q(A)=2$}
	\label{fig:AST}
\end{figure}

Let $\AST_{n,l}$ denote the set of  $(n,l)$-alternating sign trapezoids and consider $A \in \AST_{n,l}$. We introduce four different statistics on alternating sign trapezoids. First, we define
\begin{align*}
\mu(A) &\coloneqq \text{\# $-1$s in $A$,}\\
r(A) &\coloneqq \text{\# $1$-columns among the $n$ leftmost columns of $A$.}
\end{align*}
Then, we distinguish two cases: For $l \ge 2$, we define 
\begin{align*}
p(A) &\coloneqq \text{\# $10$-columns among the $n$ leftmost columns of $A$,}\\
q(A) &\coloneqq \text{\# $10$-columns among the $n$ rightmost columns of $A$.}
\end{align*}
An example of a $(4,4)$-alternating sign trapezoid is given in Figure~\ref{fig:AST}. We set the generating function $\mathcal{Z}_{\AST}(n,l;M,R,P,Q)$ of $(n,l)$-alternating sign trapezoids associated to the statistics above to be
\begin{equation*}
\sum_{A \in \AST_{n,l}}^{} M^{\mu(A)} R^{r(A)} P^{p(A)} Q^{q(A)}.
\end{equation*}
Table~\ref{tab:24ASTs} shows all $(2,4)$-alternating sign trapezoids and their respective weights. The sum of these weights yields $\mathcal{Z}_{\AST}(2,4;M,R,P,Q) = 1 + 2MR + 2R + R^2 + MP + MQ$.

\begin{table}[ht]
	\centering
	\caption{$(2,4)$-alternating sign trapezoids}
	\begin{math}
		\begin{array}{cccc}
		\toprule
		\begin{array}[t]{cccccc}
			1  &  0  &  0  &  0  &  0  &  0 \\
			&  1  &  0  &  0  &  0  &
		\end{array}
		
		&
		
		\begin{array}[t]{cccccc}
			0  &  0  &  0  &  0  &  1  &  0 \\
			&  1  &  0  &  0  &  0  &
		\end{array}
	
		&
		
		\begin{array}[t]{cccccc}
			0  &  0  &  0  &  0  &  0  &  1 \\
			&  1  &  0  &  0  &  0  &
		\end{array}
		
		&
		
		\begin{array}[t]{cccccc}
			1  &  0  &  0  &  0  &  0  &  0 \\
			&  0  &  0  &  0  &  1  &
		\end{array}
			
		\\\midrule
		
		R^2 & R Q & R & R
		
		\\\midrule\midrule
		
		\begin{array}[t]{cccccc}
			0  &  1  &  0  &  0  &  0  &  0 \\
			&  0  &  0  &  0  &  1  &
		\end{array}
		
		&
		
		\begin{array}[t]{cccccc}
			0  &  0  &  0  &  0  &  0  &  1 \\
			&  0  &  0  &  0  &  1  &
		\end{array}
		
		&
		
		\begin{array}[t]{cccccc}
			0  &  0  &  1  &  0  &  0  &  0 \\
			&  1  &  -1  &  0  &  1  &
		\end{array}
		
		&
		
		\begin{array}[t]{cccccc}
			0  &  0  &  0  &  1  &  0  &  0 \\
			&  1  &  0  &  -1  &  1  &
		\end{array}
		
		\\\midrule
		
		R P & 1 & M R & M R
		
		\\\bottomrule
		\end{array}
	\end{math}
	\label{tab:24ASTs}
\end{table}

For $l=1$, however, we adapt the statistics in the following way:
\begin{align*}
\tilde{p}(A) &\coloneqq \text{\# $10$-columns among the $n-1$ leftmost columns of $A$,}\\
\tilde{q}(A) &\coloneqq \text{\# $10$-columns among the $n-1$ rightmost columns of $A$.}
\end{align*}
We define $\mathcal{Z}_{\AST}(n,1;M,R,P,Q)$ as
\begin{equation*}
\sum_{A \in \AST_{n,1}}^{} M^{\mu(A)} R^{r(A)} P^{\tilde{p}(A)} Q^{\tilde{q}(A)} (P+Q-M)^{\left[\text{central column is a $10$-column}\right]},
\end{equation*}
where we use the \emph{Iverson bracket}: For a logical proposition $P$, $\left[P\right]=1$ if $P$ holds true and $\left[P\right]=0$ otherwise.

Table~\ref{tab:31ASTs} provides a complete list of $(3,1)$-alternating sign trapezoids and their respective weights. By adding up all these weights, we see that the generating function $\mathcal{Z}_{\AST}(3,1;M,R,P,Q)$ equals
\begin{equation}
1 - MR - MR^2 - MR^2P - MRQ + 3R + 3R^2 + R^3 + 3RP + RPQ + 3R^2P + R^2PQ + R^2P^2 + 3RQ + RQ^2 + 3R^2Q.
\end{equation}

\begin{table}[ht]
	\centering
	\caption{$(3,1)$-alternating sign trapezoids}
	\begin{math}
	\begin{array}{ccccc}
		\toprule
		\begin{array}[t]{ccccc}
			1  &  0  &  0  &  0  &  0 \\
			&  1  &  0  &  0  & \\
			&  &  1  &  & \\
		\end{array}
		
		&
		
		\begin{array}[t]{ccccc}
			0  &  0  &  0  &  1  &  0 \\
			&  1  &  0  &  0  & \\
			&  &  1  &  & \\
		\end{array}
		
		&
		
		\begin{array}[t]{ccccc}
			0  &  0  &  0  &  0 &  1 \\
			&  1  &  0  &  0  & \\
			&  &  1  &  & \\
		\end{array}
		
		&
		
		\begin{array}[t]{ccccc}
			1  &  0  &  0  &  0  &  0 \\
			&  0  &  0  &  1  & \\
			&  &  1  &  & \\
		\end{array}
	
		&
		
		\begin{array}[t]{ccccc}
			0  &  1  &  0  &  0  &  0 \\
			&  0  &  0  &  1  & \\
			&  &  1  &  & \\
		\end{array}
		
		\\\midrule
		
		R^3 & R^2 Q & R^2 & R^2 & R^2 P
		
		\\\midrule\midrule
		
		\begin{array}[t]{ccccc}
			0  &  0  &  0  &  0  &  1 \\
			&  0  &  0  &  1  & \\
			&  &  1  &  & \\
		\end{array}
		
		&
		
		\begin{array}[t]{ccccc}
			0  &  0  &  1  &  0  &  0 \\
			&  1  &  -1  &  1  & \\
			&  &  1  &  & \\
		\end{array}
		
		&
		
		\begin{array}[t]{ccccc}
			1  &  0  &  0  &  0  &  0 \\
			&  1  &  0  &  0  & \\
			&  &  0  &  & \\
		\end{array}
		
		&
		
		\begin{array}[t]{ccccc}
			0  &  0  &  1  &  0  &  0 \\
			&  1  &  0  &  0  & \\
			&  &  0  &  & \\
		\end{array}
		
		&
		
		\begin{array}[t]{ccccc}
			0  &  0  &  0  &  1  &  0 \\
			&  1  &  0  &  0  & \\
			&  &  0  &  & \\
		\end{array}
		
		\\\midrule
		
		R & M R^2 & R^2 & R^2 (P+Q-M) & R Q
		
		\\\midrule\midrule
		
		\begin{array}[t]{ccccc}
			0  &  0  &  0  &  0  &  1 \\
			&  1  &  0  &  0  & \\
			&  &  0  &  & \\
		\end{array}
	
		&
		
		\begin{array}[t]{ccccc}
			1  & 0  &  0  &  0  &  0 \\
			&  0  &  1  &  0  & \\
			&  &  0  &  & \\
		\end{array}
	
		&
	
		\begin{array}[t]{ccccc}
			0  &  1  &  0  &  0  &  0 \\
			&  0  &  1  &  0  & \\
			&  &  0  &  & \\
		\end{array}

		&

		\begin{array}[t]{ccccc}
			0  &  0  &  0  &  1  &  0 \\
			&  0  &  1  &  0  & \\
			&  &  0  &  & \\
		\end{array}

		&

		\begin{array}[t]{ccccc}
			0  &  0  &  0  &  0  &  1 \\
			&  0  &  1  &  0  & \\
			&  &  0  &  & \\
		\end{array}
		
		\\\midrule
		
		R & R^2 (P+Q-M) & R^2 P (P+Q-M) & R Q (P+Q-M) & R (P+Q-M)
		
		\\\midrule\midrule
		
		\begin{array}[t]{ccccc}
			1  &  0  &  0  &  0  &  0 \\
			&  0  &  0  &  1  & \\
			&  &  0  &  & \\
		\end{array}
		
		&
		
		\begin{array}[t]{ccccc}
			0  &  1  &  0  &  0  &  0 \\
			&  0  &  0  &  1  & \\
			&  &  0  &  & \\
		\end{array}
		
		&
		
		\begin{array}[t]{ccccc}
			0  &  0  &  1  &  0  &  0 \\
			&  0  &  0  &  1  & \\
			&  &  0  &  & \\
		\end{array}
		
		&
		
		\begin{array}[t]{ccccc}
			0  &  0  &  0  &  0  &  1 \\
			&  0  &  0  &  1  & \\
			&  &  0  &  & \\
		\end{array}

		&
		
		\begin{array}[t]{ccccc}
			0  &  0  &  1  &  0  &  0 \\
			&  1  &  -1  &  1  & \\
			&  &  0  &  & \\
		\end{array}
		
		\\\midrule
		
		R & R P & R (P+Q-M) & 1 & M R
		
		\\\bottomrule
	\end{array}
	\end{math}
	\label{tab:31ASTs}
\end{table}

Clearly, $p(A) \le r(A)$ if $l \ge 2$ and $\tilde{p}(A) \le r(A)$ if $l = 1$ for any $(n,l)$-alternating sign trapezoid $A$. Moreover, by reflecting along the vertical symmetric axis, we see that
\[\mathcal{Z}_{\AST}(n,l;M,R,P,Q) = R^n \mathcal{Z}_{\AST}(n,l;M,R^{-1},Q,P).\] 

Ayyer, Behrend and Fischer \cite{ABF} showed that $n \times n$-alternating sign matrices are equinumerous with $(n-1,3)$-alternating sign trapezoids. As a corollary of \cite[Theorem 1.2]{ABF}, the statistic~$\mu$ generalises -- in the most natural way -- a statistic on alternating sign triangles which has the same distribution as the number of $-1$s on alternating sign matrices.

\begin{definition}
	For a \emph{strict partition} $\lambda = \left( \lambda_1,\dots,\lambda_m \right)$, that is, a sequence $\lambda_1 > \dots > \lambda_m > 0$ of strictly decreasing positive integers, a \emph{shifted Young diagram} of \emph{shape} $\lambda$ is a finite collection of cells arranged in $m$ rows such that row~$i$ has length $\lambda_i$ and each row is indented by one cell compared to the row above.
	\begin{figure}[ht]
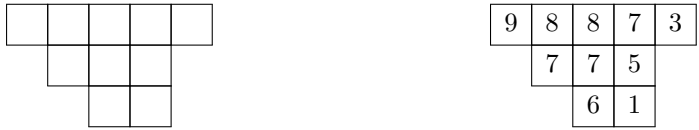

		\begin{center}
			\hfill\begin{minipage}{.4\textwidth}
				\ydiagram{5,1+3,2+2}
			\end{minipage}%
			\begin{minipage}{.4\textwidth}
				\begin{ytableau}
					9 & 8 & 8 & 7 & 3 \\
					\none & 7 & 7 & 5 \\
					\none & \none & 6 & 1 \\
				\end{ytableau}
			\end{minipage}
			\caption{A shifted Young diagram of shape~$(5,3,2)$ and a column strict shifted plane partition~$\pi$ of the same shape and of class~$4$ with $\mu_3(\pi)=2$, $r(\pi)=3$, $p_3(\pi)=1$ and $q(\pi)=1$}
			\label{fig:CSSPP}
		\end{center}
	\end{figure}

	A filling of a shifted Young diagram with positive integers such that the entries, termed \emph{parts}, weakly decrease along each row and strictly decrease down each column is called a \emph{column strict shifted plane partition}. It is of \emph{class~k} if the first part of each row~$i$ is exactly $k+\lambda_i$, that is, exactly $k$ plus its corresponding row length.
\end{definition}
	
For any $k$, we consider the collection of zero cells to be a column strict shifted plane partition with zero rows and of class $k$. Note that we usually omit the drawings of the cells. Furthermore, we cannot always associate a class to a given column strict shifted plane partition. Column strict shifted plane partitions of class $2$ correspond to \emph{descending plane partitions} as defined by Andrews \cite{And79}.

Let $\CSSPP_{n,k}$ denote the set of column strict shifted plane partitions of class $k$ with at most $n$ parts in the first row. Following the standard labelling of entries of matrices, we refer to the part in row~$i$ and column~$j$ of a column strict shifted plane partition $\pi$ as $\pi_{i,j}$. We introduce four different statistics on $\pi \in \CSSPP_{n,k}$ of which two depend on a fixed parameter~$d\in \{1,\dots,k\}$:
\begin{align*}
	\mu_d(\pi) &\coloneqq \text{\# parts $\pi_{i,j}\in\{2,3,\dots,j-i+k\} \setminus \{j-i+d\}$,}\\
	r(\pi) &\coloneqq \text{\# rows of $\pi$,}\\
	p_d(\pi) &\coloneqq \text{\# parts $\pi_{i,j}=j-i+d$,}\\
	q(\pi) &\coloneqq \text{\# parts $\pi_{i,j}=1$.}
\end{align*}
We define the generating function $\mathcal{Z}_{\CSSPP}(n,k,d;M,R,P,Q)$ associated to these statistics as
\begin{equation*}
\sum_{\pi \in \CSSPP_{n,k}}^{} M^{\mu_d(\pi)} R^{r(\pi)} P^{p_d(\pi)} Q^{q(\pi)}.
\end{equation*}

For $d=0$, we add another family of statistics by defining
\begin{align*}
	\mu_0(\pi) &\coloneqq \text{\# parts $\pi_{i,j}\in\{2,3,\dots,j-i+k\} \setminus \{j-i\}$,}\\
	p_0(\pi) &\coloneqq \text{\# parts $\pi_{i,j}=j-i>1$,}\\
	q_0(\pi) &\coloneqq \text{\# parts $\pi_{i,j}=1$ such that $j-i >1$.}
\end{align*}
In this case, the generating function $\mathcal{Z}_{\CSSPP}(n,k,0;M,R,P,Q)$ is defined as
\begin{equation*}
\sum_{\pi \in \CSSPP_{n,k}}^{} M^{\mu_0(\pi)} R^{r(\pi)} P^{p_0(\pi)} Q^{q_0(\pi)} (P+Q-M)^{\left[\text{$\pi_{i,j}=1$ such that $j-i=1$}\right]}.
\end{equation*}
Note that $j-i=1$ means that the part $\pi_{i,j}$ is in the second position of a row. A $1$ at the second position of a row can only occur in the bottom row. Also note that there is never a $1$ at the first position of a row if the class of the column strict shifted plane partition is greater than $0$.

An example of a column strict shifted plane partition is presented in Figure~\ref{fig:CSSPP}. Note that the parts counted by the statistic~$\mu$ generalise the \emph{special parts} in descending plane partitions that were first defined by Mills, Robbins and Rumsey \cite{MRR83}. They presented without proof a determinantal formula for an associated generating function that incorporates a total of three statistics. Subsequently, Behrend, Di Francesco and Zinn-Justin \cite{BFZ12} obtained a closely related determinantal formula and identified the connection to the one found by Mills, Robbins and Rumsey.

\begin{table}[ht]
	\centering
	\caption{Column strict shifted plane partitions of class~$3$ with at most two parts in the first row}
	\begin{math}
	\begin{array}{lcccccccc}
	\toprule
	
	&
	
	\begin{array}{c}
	\emptyset
	\end{array}
	
	&
	
	\begin{array}{c}
	4
	\end{array}
	
	&
	
	\begin{array}{cc}
	5 & 5
	\end{array}
	
	&
	
	\begin{array}{cc}
	5 & 4
	\end{array}
	
	&
	
	\begin{array}{cc}
	5 & 3
	\end{array}
	
	&
	
	\begin{array}{cc}
	5 & 2
	\end{array}
	
	&
	
	\begin{array}{cc}
	5 & 1
	\end{array}
	
	&
	
	\begin{array}{cc}
	5 & 5 \\
	& 4
	\end{array}
	
	\\\midrule
	
	d=0: & 1 & R & R & M R & M R & M R & R (P+Q-M) & R^2
	
	\\
	
	d=1: & 1 & R & R & M R & M R & R P & R Q & R^2
	
	\\
	
	d=2: & 1 & R & R & M R & R P & M R & R Q & R^2
	
	\\
	
	d=3: & 1 & R & R & R P & M R & M R & R Q & R^2
	
	\\\bottomrule
	\end{array}
	\end{math}
	\label{tab:23CSSPPs}
\end{table}

Table~\ref{tab:23CSSPPs} presents all column strict shifted plane partitions of class~$3$ with at most two parts in the first row. The weights with respect to all possible choices of the parameter $d\in\{0,1,2,3\}$ are listed below. By summing the weights for each $d$, we obtain $\mathcal{Z}_{\CSSPP}(2,3,d;M,R,P,Q) = 1 + 2MR + 2R + R^2 + RP + RQ$, which is independent of $d$. Similarly, the column strict shifted plane partitions of class 0 with at most three parts in the first row are shown in Table~\ref{tab:30CSSPPs}.

\begin{table}[ht]
	\centering
	\caption{Column strict shifted plane partitions of class~$0$ with at most three parts in the first row}
	\begin{math}
	\begin{array}{lccccc}
	\toprule
	
	&
	
	\begin{array}{c}
	\emptyset
	\end{array}
	
	&
	
	\begin{array}{c}
	1
	\end{array}
	
	&
	
	\begin{array}{cc}
	2 & 2
	\end{array}
	
	&
	
	\begin{array}{cc}
	2 & 1
	\end{array}
	
	&
	
	\begin{array}{ccc}
	3 & 3 & 3
	\end{array}
	
	\\\midrule
	
	d=0: & 1 & R & R & R (P+Q-M) & R 
	
	\\\midrule\midrule
	
	&
	
	\begin{array}{ccc}
	3 & 3 & 2
	\end{array}
	
	&
	
	\begin{array}{ccc}
	3 & 3 & 1
	\end{array}
	
	&
	
	\begin{array}{ccc}
	3 & 2 & 2
	\end{array}
	
	&
	
	\begin{array}{ccc}
	3 & 2 & 1
	\end{array}
	
	&
	
	\begin{array}{ccc}
	3 & 1 & 1
	\end{array}
	
	\\\midrule
	
	d=0: & R P & R Q & R P & R Q & R Q (P+Q-M)
	
	\\\midrule\midrule
	
	&
	
	\begin{array}{cc}
	2 & 2\\
	& 1
	\end{array}
	
	&
	
	\begin{array}{ccc}
	3 & 3 & 3\\
	& 1 & 
	\end{array}
	
	&
	
	\begin{array}{ccc}
	3 & 3 & 3\\
	& 2 & 2
	\end{array}
	
	&
	
	\begin{array}{ccc}
	3 & 3 & 3\\
	& 2 & 1
	\end{array}
	
	&
	
	\begin{array}{ccc}
	3 & 3 & 2\\
	& 1 &
	\end{array}
	
	\\\midrule
	
	d=0: & R^2 & R^2 & R^2 & R^2 (P+Q-M) & R^2 P
	
	\\\midrule\midrule
	
	&
	
	\begin{array}{ccc}
	3 & 3 & 2\\
	& 2 & 1
	\end{array}
	
	&
	
	\begin{array}{ccc}
	3 & 3 & 1\\
	& 1 &
	\end{array}
	
	&
	
	\begin{array}{ccc}
	3 & 2 & 2\\
	& 1 &
	\end{array}
	
	&
	
	\begin{array}{ccc}
	3 & 2 & 1\\
	& 1 & 
	\end{array}
	
	&
	
	\begin{array}{ccc}
	3 & 3 & 3\\
	& 2 & 2\\
	&   & 1
	\end{array}
	
	\\\midrule
	
	d=0: & R^2 P (P+Q-M) & R^2 Q & R^2 P & R^2 Q & R^3
	
	\\\bottomrule
	\end{array}
	\end{math}
	\label{tab:30CSSPPs}
\end{table}

It turns out that the generating functions $\mathcal{Z}_{\AST}(2,4;M,R,P,Q)$ and $\mathcal{Z}_{\CSSPP}(2,3,d;M,R,P,Q)$ for any $d \in \{0,1,2,3\}$ as well as $\mathcal{Z}_{\AST}(3,1;M,R,P,Q)$ and $\mathcal{Z}_{\CSSPP}(3,0,0;M,R,P,Q)$ coincide. Fischer \cite{Fis} established refined enumerations of alternating sign trapezoids and column strict shifted plane partitions by showing that
\begin{equation*}
\mathcal{Z}_{\AST}(n,l;1,R,P,Q) = \mathcal{Z}_{\CSSPP}(n,l-1,d;1,R,P,Q)
\end{equation*}
for any $d \in \{0,1,\dots,l-1\}$. We extend her proof by adding the fourth statistic and show that the joint distribution of the corresponding statistics on alternating sign trapezoids and on column strict shifted plane partitions coincide:

\begin{theorem}
	\label{thm:maintheorem}
	Let $n,l \ge 1$ and $0 \le d \le l-1$. Then
	\begin{equation*}
	\mathcal{Z}_{\AST}(n,l;M,R,P,Q) = \mathcal{Z}_{\CSSPP}(n,l-1,d;M,R,P,Q).
	\end{equation*}
\end{theorem}

The proof of Theorem~\ref{thm:maintheorem} is organised as follows: We provide explicit determinantal expressions for $\mathcal{Z}_{\AST}(n,l;M,R,P,Q)$ and $\mathcal{Z}_{\CSSPP}(n,l-1,d;M,R,P,Q)$ in Section~\ref{sec:EnumAST} and Section~\ref{sec:CSSPPs}, respectively. These expressions are then each found to be equal to \eqref{eq:QRST-ASTbinom}. The expression for $\mathcal{Z}_{\AST}(n,l;M,R,P,Q)$ is obtained using a bijection between alternating sign trapezoids and truncated monotone triangles, together with operator formulae for the weighted enumeration of monotone triangles due to Fischer and Riegler \cite{FR15}. The expression for $\mathcal{Z}_{\CSSPP}(n,l-1,d;M,R,P,Q)$ is obtained using a bijection between column strict shifted plane partitions and families of nonintersecting lattice paths, together with the Lindstr\"om--Gessel--Viennot lemma for the weighted enumeration of such families.

To conclude the paper, we also present a result for the $2$-enumeration of $(n,l)$-alternating sign trapezoids with $l \geq 2$ in Section~\ref{sec:2Enum}.

\section{Weighted Enumeration of Alternating Sign Trapezoids}
\label{sec:EnumAST}

First, we provide a formula for the generating function of alternating sign trapezoids. For this purpose, we heavily exploit the correspondence between alternating sign trapezoids and certain truncated \emph{monotone triangles}, where the latter are defined below.

\subsection{Correspondence between trees and alternating sign trapezoids}
\label{sec:correspondence}

\begin{definition}
	For a given integer~$n \geq 1$, a \emph{monotone triangle of order~$n$} is an array of integers in a triangular shape with $n$ rows of the following form
\begin{center}
		\begin{tikzpicture}[xscale=2]
		\mt{{$a_{n,1}$,$a_{n,2}$,$a_{n,3}$,\dots,$a_{n,n}$},
			{$a_{n-1,1}$,$a_{n-1,2}$,\dots,$a_{n-1,n-1}$},
			{\dots,\dots,\dots},
			{$a_{2,1}$,$a_{2,2}$},
			{$a_{1,1}$}}
		\end{tikzpicture}
\end{center}
	such that the entries strictly increase along all rows other than the bottom row and weakly increase both along $\nearrow$-diagonals and $\searrow$-diagonals.
\end{definition}

\begin{definition}
	For given integers $p,q \geq 0$ and $n \geq 1$ such that $p+q \le n$ as well as a weakly decreasing sequence $\mathbf{s}=(s_1,s_2,\dots,s_p)$ and a weakly increasing sequence $\mathbf{t}=(t_{n-q+1},t_{n-q+2},\allowbreak \dots,\allowbreak t_n)$ of nonnegative integers, we define an \emph{$(\mathbf{s}$,$\mathbf{t})$-tree} as an array of integers which arises from a  monotone triangle of order~$n$ by truncating the diagonals as follows: for each $1 \leq i \leq p$, we delete the $s_i$ bottom entries of the $i\textsuperscript{th}$ $\nearrow$-diagonal; for each $n-q+1 \leq i \leq n$, we delete the $t_i$ bottom entries of the $i\textsuperscript{th}$ $\searrow$-diagonal. All diagonals are counted from left to right.
	
	We say that an $(\mathbf{s}$,$\mathbf{t})$-tree has \emph{bottom row $\mathbf{k}=(k_1, \dots, k_{n})$} if the following holds true: for all $i$ such that $1 \le i \le n-q$ or $n-q+1 \le i \le n$, the integer $k_i$ is the bottom entry of the $i\textsuperscript{th}$ $\nearrow$-diagonal or the $i\textsuperscript{th}$ $\searrow$-diagonal, respectively.
\end{definition}

Figure~\ref{fig:TreeExample} gives an examples of a $((6,3,1,1),(1,4,4))$-tree. By way of illustration, we indicate the deleted parts of the truncated diagonals by $\color{gray}\bullet$. In addition, we add an edging to the shape of the tree in order to visualise the naming.

\begin{figure}[ht]
	\centering
	\begin{tikzpicture}
	\mt{{\color{gray}$\bullet$,\color{gray}$\bullet$,\color{gray}$\bullet$,\color{gray}$\bullet$,$6$,\color{gray}$\bullet$,\color{gray}$\bullet$,\color{gray}$\bullet$},
		{\color{gray}$\bullet$,\color{gray}$\bullet$,$4$,$5$,$7$,\color{gray}$\bullet$,\color{gray}$\bullet$},
		{\color{gray}$\bullet$,\color{gray}$\bullet$,$4$,$7$,\color{gray}$\bullet$,\color{gray}$\bullet$},
		{\color{gray}$\bullet$,$3$,$6$,\color{gray}$\bullet$,\color{gray}$\bullet$},
		{\color{gray}$\bullet$,$3$,$8$,$9$},
		{\color{gray}$\bullet$,$5$,$8$},
		{$1$,$7$},
		{$6$}}
	
	\draw (4.25,0.25) --++ (-2,0) --++ (1,1) --++ (-1,0) --++ (1.5,1.5) --++ (-1,0) --++ (1.25,1.25) --++ (2.25,-2.25) --++ (-2,0) --++ (1.5,-1.5) --++ (-1,0) --++ (0.5,-0.5) --++ (-1.5,0) --++ (0.5,0.5) ;
	\end{tikzpicture}
	\caption{$((6,3,1,1),(1,4,4))$-tree with eight rows ($n=8$) and bottom row $(1,3,4,5,6,7,8,9)$}
	\label{fig:TreeExample}
\end{figure}
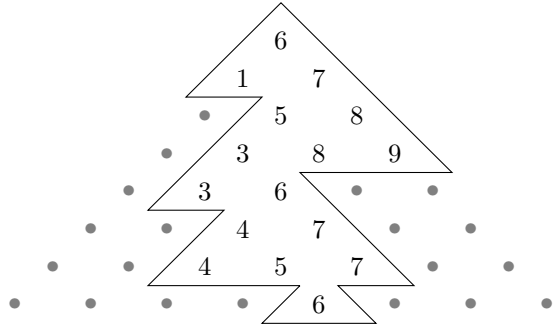

The definition of monotone triangles in the present paper differs from the original definition given by Mills, Robbins and Rumsey in \cite{MRR83}, where they also impose strict increase on the bottom row.  Although we predominantly consider strictly increasing bottom rows, monotone triangles and trees with weakly increasing bottom rows do appear in Section~\ref{sec:ASTZswithprescribedcolumnvector}, where we investigate quasi alternating sign triangles. Note that the monotonicity conditions on the diagonals of monotone triangles guarantee that the bottom row is at least weakly increasing.

Next, we show how to transform an alternating sign trapezoid into a tree: Given an $(n,l)$-alternating sign trapezoid such that $l \ge 2$, we pad the array with additional zeroes to obtain a rectangular shape of size $n \times (2n+l-2)$. To each entry, we add the entries in the same column above it. This yields an array  consisting merely of $0$s and $1$s with exactly $i$ $1$s in the $i\textsuperscript{th}$ row. We want to record the positions of these $1$s. Therefore, we number the columns from $-n$ to $n+l-3$ from left to right and list the corresponding numbers row by row in a triangular shape. By removing the entries that correspond to the initially added ones, we finally obtain a truncated monotone triangle or tree. In Figure~\ref{fig:Tree}, we exemplify the construction with the $(4,4)$-alternating sign trapezoid in Figure~\ref{fig:AST}.
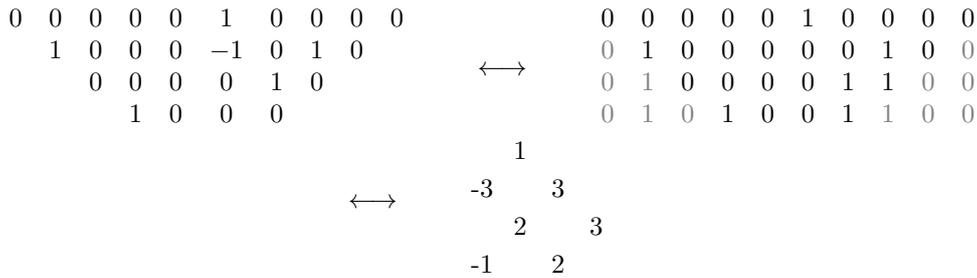
\begin{figure}[ht]
	\centering
	\begin{gather*}
	\begin{array}{cccccccccc}
	0 & 0 & 0 & 0 & 0 & 1 & 0 & 0 & 0 & 0\\
	 & 1 & 0 & 0 & 0 & -1 & 0 & 1 & 0 &  \\
	 &  & 0 & 0 & 0 & 0 & 1 & 0 &  &  \\
	 &  &  & 1 & 0 & 0 & 0 &  &  &  \\
	\end{array}
	\qquad\longleftrightarrow\qquad
	\begin{array}{cccccccccc}
	0 & 0 & 0 & 0 & 0 & 1 & 0 & 0 & 0 & 0\\
	{\color{gray}0} & 1 & 0 & 0 & 0 & 0 & 0 & 1 & 0 & {\color{gray}0} \\
	{\color{gray}0} & {\color{gray}1} & 0 & 0 & 0 & 0 & 1 & 1 & {\color{gray}0}  & {\color{gray}0} \\
	{\color{gray}0} & {\color{gray}1} & {\color{gray}0} & 1 & 0 & 0 & 1 & {\color{gray}1} & {\color{gray}0} & {\color{gray}0} \\
	\end{array}
	\\
	\longleftrightarrow
	\begin{tikzpicture}[baseline=(current bounding box),outer sep=0pt,inner sep=0pt]
		\mt{{,-1,2,},{,2,3},{-3,3},{1}}
		\end{tikzpicture}
	\end{gather*}
	\caption{$(4,4)$-alternating sign trapezoid with $1$-column vector $(-3,-1,1,2)$ and corresponding $\left((2),(1)\right)$-tree with bottom row $(-3,-1,2,3)$}
	\label{fig:Tree}
\end{figure}

To illustrate the main features of this construction, we number (differently than in Definition~\ref{def:ASTZs}) the $n$ leftmost columns of an $(n,l)$-alternating sign trapezoid from $-n$ to $-1$ and the $n$ rightmost columns from $1$ to $n$. The \emph{$1$-column vector} $\mathbf{c}=\left(c_1,\dots,c_n\right)$ records the positions of the $1$-columns of the alternating sign trapezoid; hence, $-n \le c_1 < \dots < c_m < 0 < c_{m+1} < \dots < c_n\le n$ for some $0 \le m \le n$.  The construction above yields an $(\mathbf{s}$,$\mathbf{t})$-tree with bottom row $(c_1,\dots,c_m,c_{m+1}+l-3,\dots,c_{n}+l-3)$ such that $\mathbf{s} = ( -c_1-1,\dots,-c_m-1 )$ and $\mathbf{t} = ( c_{m+1}-1,\dots,c_n-1 )$.

Regarding the statistics of alternating sign trapezoids, we make the following observations:
\begin{compactitem}
	\item A $-1$ in the alternating sign trapezoid corresponds to an entry $a_{i,j}$ in the tree which has two neighbouring entries $a_{i+1,j}$ and $a_{i+1,j+1}$ in the row below such that $a_{i+1,j} < a_{i,j} < a_{i+1,j+1}$. This observation motivates the definition of the statistic $\mu(T)$ on trees~$T$ with $n$ rows:
	\begin{equation*}
	\mu(T) \coloneqq \text{\# entries $a_{i,j}$ such that $i<n$ and $a_{i+1,j} < a_{i,j} < a_{i+1,j+1}$.}
	\end{equation*}
	\item The positions of  $1$-columns are reflected in the bottom row of the tree.
	\item $10$-columns cause the corresponding diagonals in the tree to have duplicated bottom entries.
\end{compactitem}

\subsection{Enumeration of trees}

To enumerate monotone triangles and trees, we use operator formulae and constant term expressions. To this end, we need to introduce several operators and notations. First, we define the \emph{symmetriser}~$\sym$ and the \emph{antisymmetriser}~$\asym$ of a function $f(X_1,\dots,X_n)$. Let $\mathfrak{S}_n$ be the symmetric group of degree $n$. Then 
\begin{align*}
\sym_{X_1,\dots,X_n} f(X_1,\dots,X_n) &\coloneqq \sum_{\sigma\in\mathfrak{S}_n} f(X_{\sigma(1)},\dots,X_{\sigma(n)})\,\text{and}\\
\asym_{X_1,\dots,X_n} f(X_1,\dots,X_n) &\coloneqq \sum_{\sigma\in\mathfrak{S}_n} \sgn(\sigma) f(X_{\sigma(1)},\dots,X_{\sigma(n)}).
\end{align*}
We use $\sym_{\mathbf{X}}$ and $\asym_{\mathbf{X}}$ as an abbreviation if $\mathbf{X}=(X_1,\dots,X_n)$ is clear from the context. Furthermore, $\ct_{\mathbf{X}} f(\mathbf{X})=\ct_{X_1,\dots,X_n} f(X_1,\dots,X_n)$ denotes the constant term of the function $f$ with respect to the variables $X_1,\dots,X_n$. Finally, we define the \emph{shift operator}~$\E_X$, the \emph{forward difference operator}~$\fd_X$ and the \emph{backward difference operator}~$\delta_X$:  
\begin{align*}
\E_X \left[f(X)\right] &\coloneqq f(X+1),\\
\fd_X &\coloneqq \E_X - \id,\\
\delta_X &\coloneqq \id - \E_X^{-1},
\end{align*}
where $\id$ denotes the standard identity operator. The \emph{$M$-forward difference operator}~$\Qfd_X$ and the \emph{$M$-backward difference operator}~$\Qbd_X$ are defined as follows:
\begin{align*}
\Qfd_X &\coloneqq (M\id - (1-M)\fd_X)^{-1}\fd_X,\\
\Qbd_X &\coloneqq (M\id - (M-1)\delta_X)^{-1}\delta_X.
\end{align*}
They generalise the previous difference operators as they reduce to $\fd_X$ and $\bd_X$, respectively, if we set $M=1$. We use the notation $\E_x \left[ f(x) \right] \coloneqq \left. \E_X \left[ f(X) \right] \right|_{X=x}$ for a given a variable~$X$ and an integer~$x$. This abbreviatory notation is correspondingly used for other operator expressions. Note that for a function of several variables, any of these operators which is associated with a variable~$X$ commutes with any of these operators which is associated with a different variable~$Y$.

Fischer and Riegler \cite{FR15} provided a weighted enumeration of monotone triangles:

\begin{theorem}
	\label{thm:Q-M}
	The generating function~$\mathcal{Z}_{\MT} (n, \mathbf{k};M)$ of monotone triangles $T$ of order $n$ with strictly increasing bottom row $\mathbf{k}=(k_1, \dots,\allowbreak k_{n})$ with respect to the weight~$M^{\mu(T)}$ is given by 
	\begin{equation}
	\label{eq:Q-M}
	\ct_{\mathbf{Y}} \left( \asym_{\mathbf{Y}} \left( \prod_{i=1}^n \left(1+Y_i\right)^{k_i} \prod_{1 \le i < j \le n} \left( M - (1-M)Y_i + Y_j + Y_i Y_j \right) \right) \prod_{1 \le i < j \le n} \left( Y_j - Y_i \right)^{-1} \right).
	\end{equation}
\end{theorem}

To obtain a enumeration formula for trees, we need to apply generalised difference operators to $\mathcal{Z}_{\MT} (n, \mathbf{k};M)$. The crucial observation is that if we repeatedly apply $-\fd_{k_i}$ and $\delta_{k_i}$ to $\mathcal{Z}_{\MT} (n, \mathbf{k};1)$, we enumerate monotone triangles with truncated diagonals. By using the aforementioned generalisations of the difference operators, \cite[Theorem 5]{Fis18} implies the following result:
\begin{theorem}
	\label{thm:genfunktrees}
	The generating function of $(\mathbf{s}$,$\mathbf{t})$-trees $T$ with $\mathbf{s}=(s_1,s_2,\dots,s_p)$, $\mathbf{t}=(t_{n-q+1},t_{n-q+2},\allowbreak \dots,\allowbreak t_n)$ and strictly increasing bottom row $\mathbf{k}=(k_1, \dots, k_{n})$ with respect to the weight~$M^{\mu(T)}$ is given by
	\begin{equation}
	\label{eq:genfunktrees}
	\prod_{i=1}^{p} \left( -\Qfd_{k_i} \right)^{s_i} \prod_{i=n-q+1}^{n} \Qbd_{k_i}^{t_i} \, \left[ \mathcal{Z}_{\MT} (n, \mathbf{k};M) \right].
	\end{equation}
\end{theorem}

Note that Theorems~\ref{thm:Q-M} and \ref{thm:genfunktrees} are also true for weakly increasing bottom rows $\mathbf{k}$ when we consider the straight enumeration of monotone triangles and trees by setting $M=1$.

\subsection{Enumeration of alternating sign trapezoids with prescribed \texorpdfstring{\boldmath$1$}{1}-column vector}
\label{sec:ASTZswithprescribedcolumnvector}

We use the correspondence between alternating sign trapezoids and trees to obtain enumeration formulae. First, we consider alternating sign trapezoids with prescribed $1$-column vectors. For this purpose, we make use of Theorems~\ref{thm:Q-M} and \ref{thm:genfunktrees}.

\begin{lemma}
	The generating function of $(n,l)$-alternating sign trapezoids~$A$ with $1$-column vector $\mathbf{c}$ with respect to the weight~$M^{\mu(A)} P^{p(A)} Q^{q(A)}$ is given by
	\begin{equation}
	\label{eq:QST-AST1columns}
	\prod_{i=1}^{m} \left( \id - \frac{P}{M} \delta_{c_i} \right) \left( \id + \Qfd_{c_i} \right) \left( -\Qfd_{c_i} \right)^{-c_i-1} \hspace*{-1ex} \prod_{i=m+1}^{n} \left( \id + \frac{Q}{M} \fd_{c_i} \right) \left( \id - \Qbd_{c_i} \right) \Qbd_{c_i}^{c_i-1} \left[\mathcal{Z}_{\MT} (n, \mathbf{\tilde{c}};M) \right],
	\end{equation}
	where $\mathbf{\tilde c}=\left(c_1,\dots,c_m,c_{m+1}+l-3,\dots,c_{n}+l-3\right)$.
\end{lemma}

\begin{proof}
	First, let us assume $l \geq 2$. In order to enumerate $(n,l)$-alternating sign trapezoids with $1$-column vector $\mathbf{c}$ such that $-n \le c_1 < \dots < c_m < 0 < c_{m+1} < \dots < c_n\le n$ for some $0 \le m \le n$, it suffices to enumerate $(( -c_1-1,\dots,-c_m-1 ),( c_{m+1}-1,\dots,c_n-1 ))$-trees with bottom row $\mathbf{\tilde c}=(c_1,\dots,c_m,c_{m+1}+l-3,\dots,c_{n}+l-3)$.
	
	Let $L_{10} \subseteq \{1,\dots,m\}$ and $R_{10} \subseteq \{m+1,\dots,n\}$. Consider an $(n,l)$-alternating sign trapezoid with $1$-column vector $\mathbf{c}$ such that $c_i$ is a  $10$-column if and only if $i \in L_{10} \cup R_{10}$. By the correspondence between alternating sign triangles and trees, we have to apply a generalised difference operator for every $10$-column, namely $-\Qfd_{c_i}$ if $i \in L_{10}$ and $\Qbd_{c_i}$ if $i \in R_{10}$. For every every other column, we apply $\id+\Qfd_{c_i}$ or $\id-\Qbd_{c_i}$, respectively. In total, we obtain the following generating function by Theorem~\ref{thm:genfunktrees}:
	\begin{multline}
	 \label{eq:genfunkASTcvecor}
		\prod_{i \in L_{10}} \left( -\Qfd_{c_i} \right) \prod_{\substack{1 \le i \le m,\\i \notin L_{10}}} \left( \id+\Qfd_{c_i} \right) \prod_{i=1}^{m} \left( -\Qfd \right)^{-c_i-1}\\
		\times \prod_{i \in R_{10}} \Qbd_{c_i} \prod_{\substack{m+1 \le i \le n,\\i \notin R_{10}}} \left( \id-\Qbd_{c_i} \right) \prod_{i=m+1}^{n} \Qbd^{c_i-1} \left[ \mathcal{Z}_{\MT} (n, \mathbf{\tilde{c}};M) \right] .
	\end{multline}
	Since $-\Qfd_X = -M^{-1} \bd_X (\id+\Qfd_X)$ and $\Qbd_X = M^{-1} \fd_X (\id-\Qbd_X)$, \eqref{eq:genfunkASTcvecor} is equal to
	\begin{equation*}
		\prod_{i \in L_{10}} \left( -M^{-1} \bd_{c_i} \right) \prod_{i=1}^{m} \left( \id+\Qfd_{c_i} \right) \left( -\Qfd \right)^{-c_i-1}
		\prod_{i \in R_{10}} \left( M^{-1} \fd_{c_i} \right) \prod_{i=m+1}^{n} \left( \id-\Qbd_{c_i} \right) \Qbd^{c_i-1} \left[ \mathcal{Z}_{\MT} (n, \mathbf{\tilde{c}};M) \right] .
	\end{equation*}

	As an intermediate step, we enumerate all alternating sign trapezoids with $p$ $10$-columns on the left side and $q$ on the right side, that is, $\left|L_{10}\right| = p$ and $\left|R_{10}\right| = q$. The corresponding generating function is given by
	\begin{multline*}
	e_p \left( -M^{-1}\bd_{c_1},\dots,-M^{-1}\bd_{c_m} \right)  \prod_{i=1}^{m} \left( \id+\Qfd_{c_i} \right) \left( -\Qfd \right)^{-c_i-1}\\
	\times e_q \left( M^{-1}\fd_{c_{m+1}},\dots,M^{-1}\fd_{c_n} \right) \prod_{i=m+1}^{n} \left( \id-\Qbd_{c_i} \right) \Qbd^{c_i-1} \left[ \mathcal{Z}_{\MT} (n, \mathbf{\tilde{c}};M) \right] ,
	\end{multline*}
	where $e_p$ denotes the $p\textsuperscript{th}$ elementary symmetric function. To sum over all possible positions of $10$-columns, we notice that $e_p \left( X_1,\dots,X_n \right)$ is the coefficient of $P^p$ in $\prod_{i=1}^n \left( 1+ P X_i \right)$. Hence, we finally obtain \eqref{eq:QST-AST1columns}.
	
	\begin{figure}[ht]
		\centering
		\begin{subfigure}{\textwidth}
			\begin{equation*}
				\begin{array}{ccccccccc}
					0 & 0 & 0 &  0 & 0 & 0 & 0 & 1 & 0 \\
					& 0 & 0 &  1 & 0 & 0 & 0 & 0 &   \\
					&   & 1 & -1 & 0 & 0 & 1 &   &   \\
					&   &   &  0 & 0 & 1 &   &   &   \\
					&   &   &    & 1 &   &   &   &   \\
				\end{array}
				\longleftrightarrow
				\begin{tikzpicture}[baseline=(current bounding box),outer sep=0pt,inner sep=0pt]
					\mt{{,$-1$,,,},{,$0$,,},{$-3$,$1$,},{$-2$,$2$},{$2$}}
				\end{tikzpicture}
			\end{equation*}
			\caption{$\mathbf{c}=(-3,-1,2,3,4)$\\bottom row $(-3,-1,0,1,2)$}
			\label{fig:caseA}
		\end{subfigure}\\\bigskip
		\begin{subfigure}{\textwidth}
			\begin{equation*}
				\begin{array}{ccccccccc}
					0 & 0 & 0 &  0 & 0 & 0 & 0 & 1 & 0 \\
					& 0 & 0 &  1 & 0 & 0 & 0 & 0 &   \\
					&   & 1 & -1 & 0 & 0 & 1 &   &   \\
					&   &   &  0 & 0 & 1 &   &   &   \\
					&   &   &    & 0 &   &   &   &   \\
				\end{array}
				\longleftrightarrow
				\begin{tikzpicture}[baseline=(current bounding box),outer sep=0pt,inner sep=0pt]
					\mt{{,$-1$,,,},{,$0$,,},{$-3$,$1$,},{$-2$,$2$},{$2$}}
				\end{tikzpicture}
			\end{equation*}
			\caption{$\mathbf{c}=(-3,1,2,3,4)$\\bottom row $(-3,-1,0,1,2)$}
			\label{fig:caseB}
		\end{subfigure}\\\bigskip
		\begin{subfigure}{\textwidth}
			\begin{equation*}
				\begin{array}{ccccccccc}
					0 & 0 & 0 &  0 & 0 & 0 & 0 & 1 & 0 \\
					& 0 & 0 &  1 & 0 & 0 & 0 & 0 &   \\
					&   & 1 & -1 & 1 & 0 & 0 &   &   \\
					&   &   &  0 & 0 & 1 &   &   &   \\
					&   &   &    & 0 &   &   &   &   \\
				\end{array}
				\longleftrightarrow
				\begin{tikzpicture}[baseline=(current bounding box),outer sep=0pt,inner sep=0pt]
					\mt{{,$-1$,$-1$,,},{,$-1$,$0$,},{$-3$,$-1$,},{$-2$,$2$},{$2$}}
				\end{tikzpicture}
			\end{equation*}
			\caption{$\mathbf{c}=(-3,-1,1,2,4)$\\bottom row $(-3,-1,-1,0,2)$}
			\label{fig:caseC}
		\end{subfigure}
		\caption{Three $(5,1)$-alternating sign trapezoids with corresponding trees}
		\label{fig:Examplel1}
	\end{figure}

Finally, let us have a look at the case $l=1$. We distinguish three sets of quasi alternating sign triangles which are all illustrated by an example in Figure~\ref{fig:Examplel1}: \subref{fig:caseA} those with bottom entry $1$, \subref{fig:caseB} those with bottom entry $0$ and central column sum $0$, and \subref{fig:caseC} those with bottom entry $0$ and central column sum $1$.

We generalise the construction of transforming $(n,l)$-alternating sign trapezoids into monotone triangles for the case $l \geq 2$ presented in Section~\ref{sec:correspondence} to $l=1$ as follows; this generalisation first appeared in the proof of \cite[Lemma 2.2]{Fis}. As before, we add zeroes to the quasi alternating sign triangle of order~$n$ such that we obtain a rectangular $n \times (2n-1)$ array. We number the columns from $-n$ to $n-2$ from left to right and compute the array of partial column sums. There are exactly $i$ $1$s in the $i$\textsuperscript{th} row except in the last row of the cases \subref{fig:caseB} and \subref{fig:caseC}, where the number of $1$s is $n-1$. In order to record the position of the $1$s in a triangular shape, we compensate for the missing $1$ by adding a $-1$ in the bottom row of the monotone triangle. In the case \subref{fig:caseC}, this results in violating the strict monotonicity condition of the bottom row; the rows are weakly increasing instead. These triangular arrays are called \emph{Gelfand-Tsetlin patterns}. By removing the entries that correspond to the additional added zeroes, we obtain trees as demonstrated in Figure~\ref{fig:Examplel1}.

In order to associate $\mathbf{c}$ to any quasi alternating sign triangle of order~$n$, we number the $n$ leftmost columns from $-n$ to $-1$ and the $n$ rightmost columns from $1$ to $n$ and list the position of the $1$-columns. By our numbering, the central column is equipped with two labels, namely $-1$ and $1$. In the case~\subref{fig:caseA}, we record the position of the central $1$-column by $-1$, whilst we record it twice by $-1$ and $1$ in the case~\subref{fig:caseC}. In the remaining case~\subref{fig:caseB}, we add the entry $1$ to $\mathbf{c}$ to ensure that $\mathbf{c}$ consists of $n$ entries in total. Thus, we obtain $1$-column vectors $\mathbf{c}=\left(c_1,\dots,c_n\right)$ with $-n \le c_1 < \dots < c_m < 0 < c_{m+1} < \dots < c_n\le n$ for some $0 \le m \le n$ such that $(c_1,\dots,c_m,c_{m+1}-2,\dots,c_{n}-2)$ is the bottom row of the corresponding tree.

To sum up, the cases \subref{fig:caseA}, \subref{fig:caseB} and \subref{fig:caseC} correspond to the following sets of trees: the first one corresponds to the case $c_m = -1$ and $c_{m+1} > 1$, the second one corresponds to the case $c_m < -1$ and $c_{m+1} = 1$, and the third one corresponds to the case $c_m = -1$ and $c_{m+1} = 1$. If $c_m = -1$ and $c_{m+1} = 1$, then all but two entries in the bottom row of the Gelfand-Tsetlin pattern are truncated. The two remaining entries are identically equal to $-1$ and force an entry in the upper row to be $-1$, too. In this case, \eqref{eq:genfunktrees} produces a factor $2-M$, whereas \eqref{eq:QST-AST1columns} gives a factor $P+Q-M$. This justifies the modification of the weights in the generating function of $(n,l)$-alternating sign trapezoids for $l=1$. \end{proof}

Instead of evaluating the polynomial~\eqref{eq:QST-AST1columns} at $\mathbf{\tilde c}$, we can shift the argument by suitable operators and evaluate it at $\mathbf{x}=\mathbf{0}$. Therefore, \eqref{eq:QST-AST1columns} is equal to
\begin{multline*}
\prod_{i=1}^{m} \E_{x_i}^{c_i} \left( \id - \frac{P}{M} \delta_{c_i} \right) \left( \id + \Qfd_{c_i} \right) \left( -\Qfd_{c_i} \right)^{-c_i-1}\\
\times \left. \prod_{i=m+1}^{n} \E_{x_i}^{c_i+l-3} \left( \id + \frac{Q}{M} \fd_{c_i} \right) \left( \id - \Qbd_{c_i} \right) \Qbd_{c_i}^{c_i-1} \left[ \mathcal{Z}_{\MT} (n, \mathbf{x};M) \right] \right|_{\mathbf{x}=\mathbf{0}},
\end{multline*}
which we rewrite as
\begin{multline}
\label{eq:QST-AST1columns-shifted}
\prod_{i=1}^{m} \E_{x_i}^{-1} \left( M \id -(P-M)\fd_{x_i} \right) \left( M \id -(1-M)\fd_{x_i} \right)^{c_i} \left( -\delta_{x_i} \right)^{-c_i-1}\\
\times \left. \prod_{i=m+1}^{n} \E_{x_i}^{l-2} \left( M \id +(Q-M)\delta_{x_i} \right) \left( M \id +(1-M)\delta_{x_i} \right)^{-c_i} \left( \fd_{x_i} \right)^{c_i-1} \left[ \mathcal{Z}_{\MT} (n, \mathbf{x};M) \right] \right|_{\mathbf{x}=\mathbf{0}}
\end{multline}
by using the fact that 
\begin{equation*}
	\id + \Qfd_x = M \left(\id + \fd_x\right) \left( M \id -(1-M)\fd_x \right)^{-1} = M \E_x \left( M \id -(1-M)\fd_x \right)^{-1}
\end{equation*}
and
\begin{equation*}
	\id - \Qbd_x = M \left(\id-\bd_x\right) \left( M \id +(1-M)\delta_x \right)^{-1} = M \E_x^{-1} \left( M \id +(1-M)\delta_x \right)^{-1}.
\end{equation*}
We analyse how the operators in \eqref{eq:QST-AST1columns-shifted} interact with the argument of the antisymmetriser in \eqref{eq:Q-M}: The effect of the shift operator $\E_{x_i}$ is the multiplication by $1 + Y_i$. Therefore, the application of the forward difference operator $\fd_{x_i}$ or of the backward difference operator $\delta_{x_i}$ is equivalent to the multiplication by $Y_i$ or by $Y_i(1+Y_i)^{-1}$, respectively.

This observation implies that \eqref{eq:QST-AST1columns-shifted} is equal to
\begin{multline}
\label{eq:QST-AST1columns-CT}
\ct_{\mathbf{Y}} \left( \asym_{\mathbf{Y}} \left( \prod_{i=1}^m \left(-Y_i\right)^{-c_i-1} \left(1+Y_i\right)^{c_i} \left(M-(1-M)Y_i\right)^{c_i} \left( M-(P-M)Y_i \right) \right. \right.\\
\times \prod_{i=m+1}^n Y_i^{c_i-1} \left(1+Y_i\right)^{c_i+l-3} \left(M+Y_i\right)^{-c_i} \left( M+Q Y_i \right)\\
\left. \left. \times \prod_{1 \le i < j \le n} \left( M - (1-M)Y_i + Y_j + Y_i Y_j \right) \right) \prod_{1 \le i < j \le n} \left( Y_j - Y_i \right)^{-1} \right).
\end{multline}

\subsection{Sum over all \texorpdfstring{\boldmath$1$}{1}-column vectors}

Thus far, we have considered $(n,l)$-alternating sign trapezoids with prescribed $1$-column vector $\mathbf{c}$.
%
%
To sum over all $c_i$ such that $-n \le c_1 < \dots < c_m < 0 < c_{m+1} < \dots < c_n \le n$,
%
%
we ignore the upper and lower bound in the summation since the polynomial in \eqref{eq:QST-AST1columns-CT} has no constant term if $c_1 < n$ or $c_n > n$. We use the following identity which we obtain by repeated geometric series evaluations:
\begin{equation}
\label{eq:GeometricSeries}
\sum_{0 \le x_1 < \dots < x_k} X_1^{x_1} \dots X_k^{x_k} = \prod_{i=1}^{k} X_i^{i-1} \left( 1 - \prod_{j=i}^{k} X_j \right)^{-1}.
\end{equation}
Hence, by applying \eqref{eq:GeometricSeries}, we obtain that
\begin{multline*}
\sum_{c_1 < \dots < c_m < 0} \left( \prod_{i=1}^m \left(-Y_i\right)^{-c_i-1} \left(1+Y_i\right)^{c_i} \left(M-(1-M)Y_i\right)^{c_i} \left( M-(P-M)Y_i \right) \right)\\
\times \sum_{0 < c_{m+1} < \dots < c_n} \left( \prod_{i=m+1}^n Y_i^{c_i-1} \left(1+Y_i\right)^{c_i+l-3} \left(M+Y_i\right)^{-c_i} \left( M+Q Y_i \right) \right)\\
\times \prod_{1 \le i < j \le n} \left( M - (1-M)Y_i + Y_j + Y_i Y_j \right) \left( Y_j - Y_i \right)^{-1}.
\end{multline*}
is equal to
\begin{multline}
	\label{eq:QST-ASTsumOverC}
\prod_{i=1}^m \frac{1}{1+Y_i} \left( \frac{-Y_i}{\left( 1+Y_i \right) \left( M-(1-M)Y_i \right)} \right)^{m-i} \\
\times \left( 1 - \prod_{j=1}^i \left( \frac{-Y_j}{\left( 1+Y_j \right) \left( M-(1-M)Y_j \right)} \right) \right)^{-1} \frac{M-(P-M)Y_i}{M-(1-M)Y_i}\\
\times \prod_{i=m+1}^n \left( 1+Y_i \right)^{l-2} \left( \frac{Y_i \left( 1+Y_i \right)}{M+Y_i} \right)^{i-m-1} \left( 1 - \prod_{j=i}^n \left( \frac{Y_j \left( 1+Y_j \right)}{M+Y_j} \right) \right)^{-1} \frac{M+T Y_i}{M+Y_i}\\
\times \prod_{1 \le i < j \le n} \left( M - (1-M)Y_i + Y_j + Y_i Y_j \right) \left( Y_j - Y_i \right)^{-1}.
\end{multline}
Before summing over all $m$ such that $0 \le m \le n$, we apply the symmetriser to the expression~\eqref{eq:QST-ASTsumOverC}. To this end, we use the following trick by Fischer \cite{Fis}: We set
\begin{equation*}
\mathfrak{S}_n^m \coloneqq \{ \sigma \in \mathfrak{S}_n \mid \sigma (i) < \sigma (j) \text{ for all } 1 \le i < j \le m \text{ and all } m+1 \le i < j \le n \}
\end{equation*}
and define
\begin{equation*}
\subsets_{X_1,\dots,X_m}^{X_{m+1},\dots,X_n} f(X_1,\dots,X_n) \coloneqq \sum_{\sigma \in \mathfrak{S}_n^m} f\left( X_{\sigma(1)},\dots,X_{\sigma(n)} \right).
\end{equation*}
It follows that
\begin{equation*}
\sym_{X_1,\dots,X_n} f(X_1,\dots,X_n) = \subsets_{X_1,\dots,X_m}^{X_{m+1},\dots,X_n} \sym_{X_1,\dots,X_m} \sym_{X_{m+1},\dots,X_n} f(X_1,\dots,X_n).
\end{equation*}
That is, we first apply $\sym_{Y_1,\dots,Y_m}$ and $\sym_{Y_{m+1},\dots,Y_n}$ to \eqref{eq:QST-ASTsumOverC} by means of the following antisymmetriser lemma \cite[Lemma~5.9]{Hon} and its variation \cite[(5.19)]{Hon} by setting $X_i \mapsto -(X_{m+1-i}) (1+X_{m+1-i})^{-1}$:
\begin{lemma}
	Let $n \ge 1$. Then
	\begin{multline*}
	\asym_{\mathbf{X}} \left( \prod_{i=1}^{n} \frac{\left( \frac{X_i(1+X_i)}{M+X_i} \right)^{i-1}}{1-\prod_{j=i}^{n} \frac{X_j(1+X_j)}{M+X_j} } \prod_{1\leq i<j \leq n} ( M - (1-M) X_i + X_j + X_i X_j ) \right)\\
	= \prod_{i=1}^{n} \frac{M+X_i}{M - X_i^2} \prod_{1\leq i<j \leq n} \frac{(M(1+X_i)(1+X_j)-X_i X_j)(X_j-X_i)}{M-X_i X_j}.
	\end{multline*}
	%
\end{lemma}

Eventually, we obtain
\begin{multline}
\label{eq:QST-ASTsym}
\prod_{i=1}^m \frac{M-(P-M)Y_i}{M(1+Y_i)^2-Y_i^2} \prod_{1 \le i < j \le m} \frac{M-Y_i Y_j}{M(1+Y_i)(1+Y_j)-Y_i Y_j} \prod_{i=m+1}^n \left( 1+Y_i \right)^{l-2} \frac{M+Q Y_i}{M-Y_i^2}\\
\times \prod_{m+1 \le i < j \le n} \frac{M(1+Y_i)(1+Y_j)-Y_i Y_j}{M-Y_i Y_j} \prod_{i=1}^m \prod_{j=m+1}^n \frac{M-(1-M)Y_i+Y_j+Y_i Y_j}{Y_j-Y_i}.
\end{multline}
Next, we need to apply the operator $\subsets_{Y_1,\dots,Y_m}^{Y_{m+1},\dots,Y_n}$ to \eqref{eq:QST-ASTsym} and evaluate at $\mathbf{Y} = 0$. To simplify the computation, we divide \eqref{eq:QST-ASTsym} by the polynomial $\prod_{1 \le i < j \le n} ( M(1+Y_i)(1+Y_j)-Y_i Y_j ) ( M-Y_i Y_j )$, which is symmetric and, thus, invariant under the application of $\subsets_{Y_1,\dots,Y_m}^{Y_{m+1},\dots,Y_n}$. However, we need to incorporate its constant term $M^{n(n-1)}$. We obtain
\begin{multline}
\label{eq:QST-ASTsymSimplified}
M^{n(n-1)} \prod_{i=1}^m \left( M-(P-M)Y_i \right) \prod_{i,j=1}^{m} \frac{1}{M(1+Y_i)(1+Y_j)-Y_i Y_j} \prod_{i=m+1}^n \left( 1+Y_i \right)^{l-2} \left( M+Q Y_i \right)\\
\times \prod_{i,j=m+1}^{n} \frac{1}{M-Y_i Y_j} \prod_{i=1}^m \prod_{j=m+1}^n \frac{M-(1-M)Y_i+Y_j+Y_i Y_j}{\left( Y_j-Y_i \right) \left( M(1+Y_i)(1+Y_j)-Y_i Y_j \right) \left( M-Y_i Y_j \right)}.
\end{multline}
This expression can be written in determinantal form. For this purpose, we consider the \emph{Cauchy determinant}
\begin{equation*}
\det_{1 \le i,j \le n} \left( \frac{1}{X_i + Y_j} \right) = \frac{\prod_{1 \le i < j \le n} \left( X_j - X_i \right) \left( Y_j - Y_i \right)}{\prod_{i,j=1}^{n} \left( X_i + Y_j \right)}
\end{equation*}
and set $X_i = \frac{M(1+Y_i)}{M-(1-M)Y_i}$ for all $1 \le i \le m$ and $X_i = -\frac{M}{Y_i}$ for all $m+1 \le i \le n$. This yields that
\begin{equation}\label{eq:insertedDet}
\det_{1 \le i,j \le n} \left(\begin{cases}
\frac{M-(1-M)Y_i}{M(1+Y_i)(1+Y_j) - Yi Y_j}, & 1 \le i \le m\\
\frac{-Y_i}{M - Y_i Y_j}, & m+1 \le i \le n
\end{cases}\right)
\end{equation}
is equal to
\begin{multline*}
(-1)^{n-m} M^{\binom{n}{2}} \prod_{i=1}^m \left( M-(1-M)Y_i \right) \prod_{i,j=1}^{m} \frac{1}{M(1+Y_i)(1+Y_j)-Y_i Y_j} \prod_{1 \le i < j \le m} \left( Y_j - Y_i \right)^2\\
\times \prod_{i=m+1}^n Y_i \prod_{i,j=m+1}^{n} \frac{1}{M-Y_i Y_j} \prod_{m+1 \le i < j \le n} \left( Y_j - Y_i \right)^2\\
\times \prod_{i=1}^m \prod_{j=m+1}^n \frac{\left( Y_j-Y_i \right) \left( M-(1-M)Y_i+Y_j+Y_i Y_j \right)}{\left( M(1+Y_i)(1+Y_j)-Y_i Y_j \right) \left( M-Y_i Y_j \right)}.
\end{multline*}
Next, we perform the following row operations on the determinant's underlying matrix in \eqref{eq:insertedDet}: For $1 \leq i \leq m$, we multiply the $i\textsuperscript{th}$ row by $(M-(P-M)Y_i)/(M-(1-M)Y_i)$, whereas, for $m+1 \leq i \leq n$, we multiply the $i\textsuperscript{th}$ row by $-(1+Y_i)^{l-2} (M+Q Y_i)/Y_i$. Finally, by multiplying the entire determinant by $M^{\binom{n}{2}}/\prod_{1 \le i < j \le n} \left( Y_j - Y_i \right)^2$, we see that \eqref{eq:QST-ASTsymSimplified} is equal to
\begin{equation*}
\frac{M^{\binom{n}{2}}}{\prod_{1 \le i < j \le n} \left( Y_j - Y_i \right)^2} \det_{1 \le i,j \le n} \left(\begin{cases}
\frac{M-(P-M)Y_i}{M(1+Y_i)(1+Y_j) - Yi Y_j}, & 1 \le i \le m\\
\left( 1+Y_i \right)^{l-2} \frac{M + Q Y_i}{M - Y_i Y_j}, & m+1 \le i \le n
\end{cases}\right).
\end{equation*}
%

To apply the subset operator, we observe that the antisymmetry of the determinant implies that it follows for any $\sigma \in \mathfrak{S}_n^m$ that
\begin{equation*}
\frac{M^{\binom{n}{2}}}{\prod_{1 \le i < j \le n} \left( Y_j - Y_i \right)^2} \det_{1 \le i,j \le n} \left(\begin{cases}
\frac{M-(P-M)Y_i}{M(1+Y_i)(1+Y_j) - Yi Y_j}, & i \in \sigma \left(\{1,\dots,m\}\right)\\
\left( 1+Y_i \right)^{l-2} \frac{M + Q Y_i}{M - Y_i Y_j}, & i \in \sigma \left(\{m+1,\dots,n\}\right)
\end{cases}\right)\\
\end{equation*}
is equivalent to applying $\sigma$ to the variables~$Y_1,\dots,Y_n$ of \eqref{eq:QST-ASTsymSimplified}. Hence, the application of $\subsets_{Y_1,\dots,Y_m}^{Y_{m+1},\dots,Y_n}$ and the summation over all $1 \le m \le n$ yield
\begin{equation}
\label{eq:QRST-ASTdet}
\frac{M^{\binom{n}{2}}}{\prod_{1 \le i < j \le n} \left( Y_j - Y_i \right)^2} \det_{1 \le i,j \le n} \left(
R \frac{M-(P-M)Y_i}{M(1+Y_i)(1+Y_j) - Y_i Y_j} + \left( 1+Y_i \right)^{l-2} \frac{M + Q Y_i}{M - Y_i Y_j}
\right),
\end{equation}
where the insertion of the factor~$R$ into the first term of the matrix entries gives the required factors of $R^m$ in the sum over $m$.

\subsection{Transformation of the determinantal formula}

The constant term $\ct_{\mathbf{Y}}$ of the determinant expression in \eqref{eq:QRST-ASTdet} provides our first expression for the generating function $\mathcal{Z}_{\AST}(n,l;M,R,P,Q)$ for the fourfold refined enumeration of $(n,l)$-alternating sign trapezoids. We transform it into a determinant involving binomial coefficients. Our key tool is the following formula by Behrend, Di Francesco and Zinn-Justin \cite[(43)-(47)]{BFZ12}:
\begin{lemma}
		\label{lem:BFZtrick}
For a given power series $f$ in variables~$X$ and $Y$, it holds that
\begin{equation*}
	\label{eq:BFZtrick}
\left. \frac{\det_{1 \le i,j \le n} \left( f(X_i,Y_j) \right)}{\prod_{1 \le i < j \le n} \left( X_j - X_i \right) \left( Y_j - Y_i \right)} \right|_{\mathbf{X} = \mathbf{Y} = \mathbf{0}} = \det_{0 \le i,j \le n-1} \left( \left[ X^i Y^j \right] f(X,Y) \right);
\end{equation*}
here, $\left[ X^i Y^j \right] f(X,Y)$ denotes the coefficient of $X^i Y^j$ in the series expansion of $f$.
\end{lemma}

We set
\begin{equation*}
f(X,Y) = R \frac{M-(P-M)X}{M(1+X)(1+Y) - X Y} + \left( 1+X \right)^{l-2} \frac{M + Q X}{M - X Y}
\end{equation*}
and extract the coefficient $\left[ X^i Y^j \right] f(X,Y)$. It equals $R (-1)^{j} + \left[j=0\right]$ if $i=0$; otherwise, $\left[ X^i Y^j \right] f(X,Y)$ is given by
\begin{equation*}
R (-1)^{i+j} \sum_{k \ge 0}  \binom{j}{k} M^{-k} \left( \binom{i-1}{k-1} + \binom{i-1}{k} P M^{-1} \right) + \binom{l-2}{i-j} M^{-j} + \binom{l-2}{i-j-1} Q M^{-j-1}.
\end{equation*}
Note that we set the binomial coefficient $\binom{n}{k} \coloneqq 0$ for $k<0$.

The determinant of the matrix $\left( \left[ X^i Y^j \right] f(X_i,Y_j) \right)_{0 \le i,j \le n-1}$ remains invariant under the left multiplication by the triangular matrix $\left( \binom{2-l}{i-j} \right)_{0 \le i,j \le n-1}$ with determinant $1$. Thus, we obtain
\begin{equation*}
R \sum_{k \ge 0}  (-1)^{j-k} \binom{j}{k} M^{-k} \left( \binom{2-l-k}{i-k} - \binom{1-l-k}{i-k-1} P M^{-1} \right)  + M^{-j} \delta_{i,j} + Q M^{-j-1} \delta_{i,j+1}.
\end{equation*}
In addition, we multiply the $i\textsuperscript{th}$ row by $(-1)^i$ and the $j\textsuperscript{th}$ row by $(-1)^j$ for all $0 \le i,j \le n-1$ and use $\binom{n}{k} = (-1)^{k} \binom{k-n-1}{k}$ to obtain
\begin{equation*}
R \sum_{k \ge 0} \binom{j}{k} M^{-k} \left( \binom{i+l-3}{i-k} + \binom{i+l-3}{i-k-1} P M^{-1} \right) + M^{-j} \delta_{i,j} - Q M^{-j-1} \delta_{i,j+1}.
\end{equation*}
Next, define the matrix
\begin{equation*}
K(n) \coloneqq \left( M^{-j} \delta_{i,j} - M^{-j-1} Q \delta_{i,j+1} \right)_{0 \le i,j \le n-1},
\end{equation*}
whose determinant evaluation yields $M^{-\binom{n}{2}}$ and whose inverse is
\begin{equation*}
K(n)^{-1} =%
\begin{cases}
Q^{i-j} M^{j}, & i \ge j,\\
0, & i < j.
\end{cases}
\end{equation*}
We left multiply by $K(n)^{-1}$, and, thus, we finally obtain by Lemma~\ref{lem:BFZtrick} that 
$\mathcal{Z}_{\AST}(n,l;Q,R,S,T)$ is equal to
\begin{equation}
\label{eq:QRST-ASTbinom}
\det_{0 \le i,j \le n-1} \left(R \sum_{k = 0}^{i} Q^{i-k} \sum_{m=0}^{j} \binom{j}{m} M^{k-m} \left( \binom{k+l-3}{k-m} + \binom{k+l-3}{k-m-1} P M^{-1} \right) + \delta_{i,j} \right).
\end{equation}

\section{Weighted Enumeration of Column Strict Shifted Plane Partitions}
\label{sec:CSSPPs}

In order to enumerate column strict shifted plane partitions, we interpret them as a family of nonintersecting lattice paths which can be enumerated by the Lindstr{\"o}m--Gessel--Viennot lemma.

\subsection{Interpretation of column strict shifted plane partitions as families of nonintersecting lattice paths}

We transform column strict shifted plane partitions into a family of nonintersecting lattice paths as follows: Each row corresponds to a path that consists of vertical and horizontal unit steps. If $(\pi_{i,i},\pi_{i,i+1},\dots,\allowbreak \pi_{i,i+\lambda_i-1})$ denotes the $i\textsuperscript{th}$ row of a column strict shifted plane partition $\pi$ of shape $\lambda$, then the associate path starts at $(\lambda_i-1,0)$ on the $x$-axis and ends at $(-1,\pi_{i,i}-1)$; the heights of the vertical steps are the parts of the corresponding row in reverse order and reduced by $1$, that is, $\pi_{i,i+\lambda_i-1}-1$,\dots , $\pi_{i,i+1}-1$ and $\pi_{i,i}-1$. Figure~\ref{fig:LatticePaths} displays the family of nonintersecting lattice paths corresponding to the column strict shifted plane partition in Figure~\ref{fig:CSSPP}. 
\begin{figure}[ht]
	\centering
		\begin{tikzpicture}[scale=.5]
		\draw [help lines,step=1cm, dashed] (-1.75,-.75) grid (4.75,8.75);
		
		\draw[->,thick] (-1.75,0)--(4.75,0) node[right]{$x$};
		\draw[->,thick] (0,-.75)--(0,8.75) node[above]{$y$};
		
		\fill (4,0) circle (5pt);
		\fill (2,0) circle (5pt);
		\fill (1,0) circle (5pt);
		
		\fill (0,8) circle (5pt);
		\fill (0,6) circle (5pt);
		\fill (0,5) circle (5pt);
		
		
		\draw[ultra thick] (-1,8) -- (0,8) -- (0,7) -- (2,7) -- (2,6) -- (3,6) -- (3,2) -- (4,2) -- (4,0);
		
		\node at (-0.5,8.5) {$9$};
		\node at (0.5,7.5) {$8$};
		\node at (1.5,7.5) {$8$};
		\node at (2.5,6.5) {$7$};
		\node at (3.5,2.5) {$3$};
		
		\draw[ultra  thick] (-1,6) -- (1,6) -- (1,4) -- (2,4) -- (2,0);
		
		\node at (-0.5,6.5) {$7$};
		\node at (0.5,6.5) {$7$};
		\node at (1.5,4.5) {$5$};
		
		\draw[ultra thick] (-1,5) -- (0,5) -- (0,0) -- (1,0);
		
		\node at (-0.5,5.5) {$6$};
		\node at (0.5,0.5) {$1$};
		
		\end{tikzpicture}
		\caption{Family of lattice paths corresponding to the corresponding to the column strict shifted plane partition in Figure~\ref{fig:CSSPP}}
		\label{fig:LatticePaths}
\end{figure}
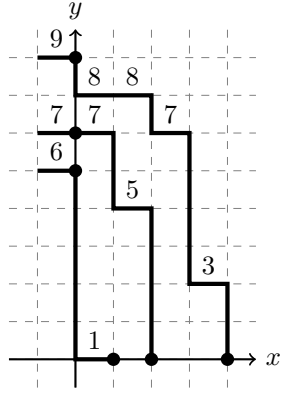

This construction yields a bijective correspondence between column strict shifted plane partitions of class $l-1$ with at most $n$ parts in the first row and the family of nonintersecting lattice paths consisting of horizontal and vertical unit steps with starting points $S \subseteq \{ S_i \coloneqq (i,0) \mid  0 \le i \le n-1\}$ and endpoints $E \subseteq \{ E_i \coloneqq (0,i+l-1) \mid  0 \le i \le n-1\}$ such that $S_i \in S$ if and only if $E_i \in E$.

By this interpretation of column strict shifted plane partitions as a family of nonintersecting lattice paths and by the Lindstr{\"o}m--Gessel--Viennot lemma, it can be shown that column strict shifted plane partitions of class $l-1$ with at most $n$ parts in the first row are enumerated by
\begin{equation}
\label{eq:AndrewsDet}
\det_{0 \le i,j \le n-1} \left( \binom{i+j+l-1}{i} + \delta_{i,j} \right).
\end{equation}
This was first proved by Andrews \cite{And79}. In fact, the determinant~\eqref{eq:AndrewsDet} can be obtained from \eqref{eq:QRST-ASTbinom} by setting $M=R=P=Q=1$.

\subsection{Weighted enumeration for {\boldmath$d \ge 1$}}

We generalise Andrews' result. To this end, consider an element of $\CSSPP_{n,l-1}$ and $L$ be the corresponding family of nonintersecting lattice paths. For the time being, fix $d \in \{1,\dots,l-1\}$. We investigate how to translate the statistics on column strict shifted plane partitions into statistics on lattice paths: The weight of $L$ is
\begin{equation*}
M^{\mu_d(L)} R^{r(L)} P^{p_d(L)} Q^{q(L)},
\end{equation*}
where
\begin{align*}
\mu_d(L) &\coloneqq \text{\# horizontal steps of $L$ below the line $y=x+l-1$ neither touching the line $y=x+d$}\\
&\hphantom{\coloneqq\qquad} \text{from the right nor lying on the $x$-axis,}\\
r(L) &\coloneqq \text{\# paths of $L$,}\\
p_d(L) &\coloneqq \text{\# horizontal steps of $L$ touching the line $y=x+d$ from the right,}\\
q(L) &\coloneqq \text{\# horizontal steps of $L$ on the $x$-axis.}
\end{align*}

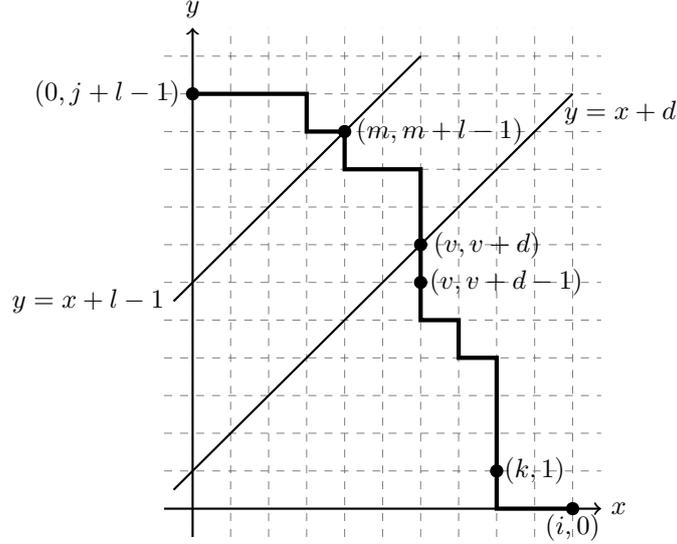
\begin{figure}[ht]
	\centering
	\begin{tikzpicture}[scale=.5]
	\draw [help lines,step=1cm, dashed] (-.75,-.75) grid (10.75,12.75);
	
	\draw[->,thick] (-.75,0)--(10.75,0) node[right]{$x$};
	\draw[->,thick] (0,-.75)--(0,12.75) node[above]{$y$};
	
	\fill (10,0) circle (5pt);
	\node at (10,-.5) {\contour{white}{$(i,0)$}};
	
	\fill (8,1) circle (5pt);
	\node at (9,1) {\contour{white}{$(k,1)$}};
	
	\fill (6,6) circle (5pt);
	\node at (8.25,6) {\contour{white}{$(v,v+d-1)$}};
	
	\fill (6,7) circle (5pt);
	\node at (7.75,7) {\contour{white}{$(v,v+d)$}};
	
	\fill (4,10) circle (5pt);
	\node at (6.5,10) {\contour{white}{$(m,m+l-1)$}};
	
	\fill (0,11) circle (5pt);
	\node at (-2.25,11) {\contour{white}{$(0,j+l-1)$}};
	
	\draw[ultra thick] (0,11) --++ (1,0) --++ (1,0) --++ (1,0) --++ (0,-1) --++ (1,0) --++ (0,-1) --++ (1,0) --++ (1,0) --++ (0,-1) --++ (0,-1) --++ (0,-1) --++ (0,-1) --++ (1,0) --++ (0,-1) --++ (1,0) --++ (0,-1) --++ (0,-1) --++ (0,-1) --++ (0,-1) --++ (1,0) --++ (1,0);
	
	\draw[thick] (-.5,5.5) -- (6,12);
	\node at (-2.75,5.5) {\contour{white}{$y=x+l-1$}};
	
	\draw[thick] (-.5,.5) -- (10,11);
	\node at (11.25,10.5) {\contour{white}{$y=x+d$}};
	
	\end{tikzpicture}
	\caption{Lattice path from $(i,0)$ to $(0,j+l-1)$ with $i-k$ horizontal steps at height $0$ that intersects the line~$y=x+d$ from the right side with a vertical step}
	\label{fig:LatticePathsStatistics}
\end{figure}

Figure~\ref{fig:LatticePathsStatistics} shows a single lattice path from $(i,0)$ to $(0,j+l-1)$ that has $i-k$ horizontal steps at height~$0$. In addition, the line $y=x+d$ intersects the path with a vertical step from the right side. Let $\mathcal{P}( (x_1,y_1) \rightarrow (x_2,y_2) )$ denote the number of paths from $(x_1,y_1)$ to $(x_2,y_2)$ with horizontal steps $(-1,0)$ and vertical steps $(0,1)$; that is, $\mathcal{P}( (x_1,y_1) \rightarrow (x_2,y_2) ) = \binom{(x_1-x_2)+(y_2-y_1)}{x_1-x_2}$. We extend the previously defined statistics $\mu_d$, $r$, $p_d$ and $q$ to single paths from $(i,0)$ to $(0,j+l-1)$ and derive the generating function of all paths from $(i,0)$ to $(0,j+l-1)$ which intersect the line $y=x+d$ with a vertical step from the right as
\begin{multline*}
R \sum_{k=0}^{i} Q^{i-k} \sum_{v=0}^{k} \mathcal{P} \left((k,1) \rightarrow (v,v+d-1)\right) \\ \times\sum_{m=0}^{j} M^{k-m} \mathcal{P} \left((v,v+d) \rightarrow (m,m+l-1)\right) \mathcal{P} \left((m,m+l-1) \rightarrow (0,j+l-1)\right)\\
=R \sum_{k=0}^{i} Q^{i-k} \sum_{v=0}^{k} \binom{k+d-2}{k-v} \sum_{m=0}^{j} M^{k-m} \binom{l-1-d}{v-m} \binom{j}{m},
\end{multline*}
which we simplify to
\begin{equation}
	\label{eq:GenFuncCSSPPwithoutS}
R \sum_{k=0}^{i} Q^{i-k} \sum_{m=0}^{j} M^{k-m} \binom{j}{m} \binom{k+l-3}{k-m}
\end{equation}
by using the Chu-Vandermonde identity.

If the line~$y=x+d$ intersects the path with a horizontal step from the right side, the generating function of such paths is given by
\begin{multline*}
R P \sum_{k=0}^{i} Q^{i-k} \sum_{v=0}^{k} \mathcal{P} \left((k,1) \rightarrow (v+1,v+d)\right) \\
\times\sum_{m=0}^{j} M^{k-m-1} \mathcal{P} \left((v,v+d) \rightarrow (m,m+l-1)\right) \mathcal{P} \left((m,m+l-1) \rightarrow (0,j+l-1)\right),
\end{multline*}
which is equal to
\begin{equation}
	\label{eq:GenFuncCSSPPwithS}
R P\sum_{k=0}^{i} Q^{i-k} \sum_{m=0}^{j} M^{k-m-1} \binom{j}{m} \binom{k+l-3}{k-m-1}.
\end{equation}
The sum of \eqref{eq:GenFuncCSSPPwithoutS} and \eqref{eq:GenFuncCSSPPwithS} yields the generating function of all single lattice paths from $(i,0)$ to $(0,j+l-1)$:
\begin{equation}
	\label{eq:GenFuncPath}
R \sum_{k = 0}^{i} Q^{i-k} \sum_{m=0}^{j} \binom{j}{m} M^{k-m} \left( \binom{k+l-3}{k-m} + \binom{k+l-3}{k-m-1} P M^{-1} \right).
\end{equation}
By the well-known Lindstr{\"o}m--Gessel--Viennot lemma \cite{GV85,GV89,Lin73}, $\mathcal{Z}_{\CSSPP}(n,l-1,d;M,R,P,Q)$ is equal to
\begin{equation*}
\sum_{\substack{0 \le u_1 < \dots < u_r \le n-1,\\0 \le r \le n-1}}^{} \det_{1 \le i,j \le r} \left( R \sum_{k = 0}^{u_i} Q^{u_i-k} \sum_{m=0}^{u_j} \binom{u_j}{m} M^{k-m} \left( \binom{k+l-3}{k-m} + \binom{k+l-3}{k-m-1} P M^{-1} \right) \right),
\end{equation*}
which simplifies to \eqref{eq:QRST-ASTbinom}.

\subsection{The case {\boldmath$d = 0$}}

In the case $d=0$, a family~$L$ of nonintersecting lattics paths is assigned the weight
\begin{equation*}
M^{\mu_0(L)} R^{r(L)} P^{p_0(L)} Q^{q_0(L)} (P+Q-M)^{\left[ \text{$L$ contains $(1,0) \rightarrow (0,0)$} \right]},
\end{equation*} 
where
\begin{align*}
\mu_0(L) &\coloneqq \text{\# horizontal steps of $L$ below the line $y=x+l-1$ neither touching the line $y=x$}\\
&\phantom{\coloneqq\qquad} \text{from the right nor lying on the $x$-axis,}\\
p_0(L) &\coloneqq \text{\# horizontal steps of $L$ touching the line $y=x$ from the right but not on the $x$-axis,}\\
q_0(L) &\coloneqq \text{\# horizontal steps of $L$ on the $x$-axis except $(1,0) \rightarrow (0,0)$.}
\end{align*}

To derive the generating function of all single lattice paths from $(i,0)$ to $(0,j+l-1)$, we distinguish several cases: If the path contains the horizontal step $(1,0) \rightarrow (0,0)$, it has the weight $R Q^{i-1} (P+Q-M)$. If the path does not contain the horizontal step $(1,0) \rightarrow (0,0)$, it 
can intersect the line $y=x$ with a vertical step from below. In this case, the generating function is given by
\begin{multline*}
R \sum_{k=1}^{i} Q^{i-k} \sum_{v=1}^{k} \mathcal{P} \left((k,1) \rightarrow (v,v-1)\right) \\ \times\sum_{m=0}^{j} M^{k-m} \mathcal{P} \left((v,v) \rightarrow (m,m+l-1)\right) \mathcal{P} \left((m,m+l-1) \rightarrow (0,j+l-1)\right),
\end{multline*}
which is equal to
\begin{equation}
	\label{eq:GenFuncCSSPPwithoutSd}
R Q^{i-1} M + R \sum_{k=1}^{i} Q^{i-k} \sum_{m=0}^{j} M^{k-m} \binom{j}{m} \binom{k+l-3}{k-m}.
\end{equation}
If the path intersects the line $y=x$ with a horizontal step from the right side, the generating function is given by
\begin{multline*}
R P \sum_{k=2}^{i} Q^{i-k} \sum_{v=1}^{k} \mathcal{P} \left((k,1) \rightarrow (v+1,v)\right) \\ \times\sum_{m=0}^{j} M^{k-m-1} \mathcal{P} \left((v,v) \rightarrow (m,m+l-1)\right) \mathcal{P} \left((m,m+l-1) \rightarrow (0,j+l-1)\right),
\end{multline*}
which is equal to
\begin{equation}
	\label{eq:GenFuncCSSPPwithSd}
R P \sum_{k=2}^{i} Q^{i-k} \sum_{m=0}^{j} M^{k-m-1} \binom{j}{m} \binom{k+l-3}{k-m-1}.
\end{equation}
As before, the sum of \eqref{eq:GenFuncCSSPPwithoutSd}, \eqref{eq:GenFuncCSSPPwithSd} and $R Q^{i-1} (P+Q-M)$ is equal to \eqref{eq:GenFuncPath}.

As a result, \eqref{eq:QRST-ASTbinom} is the generating function $\mathcal{Z}_{\CSSPP}(n,l-1,d;M,R,P,Q)$ of column strict shifted plane partitions of class $l-1$ with at most $n$ entries in the first row for any $d\in\{0,\dots,l-1\}$. This completes the proof of Theorem~\ref{thm:maintheorem}.



\section{The \texorpdfstring{\boldmath$x$}{x}-Enumeration}
\label{sec:2Enum}

For an integer $x\in\mathbb{Z}$, we define the \emph{$x$-enumeration} of $(n,l)$-alternating sign trapezoids as
\begin{equation*}
\sum_{A\in\AST_{n,l}} x^{\mu(A)}.
\end{equation*}
Ayyer, Behrend and Fischer \cite{ABF} proved that the analogously defined $x$-enumeration of $n \times n$-alternating sign matrices is equal to the $x$-enumeration of $(n-1,3)$-alternating sign trapezoids. For specific values of $x$, the $x$-enumeration of alternating sign matrices surprisingly yields round numbers: Mills, Robbins and Rumsey \cite{MRR83} presented a closed expression of the $2$-enumeration and Kuperberg \cite{Kup96} for the $3$-enumeration.

We prove that the $2$-enumeration of $(n,l)$-alternating sign trapezoids yields round numbers in general, too.

\begin{theorem}
	\label{thm:2Enum}
	For $l \ge 2$, the $2$-enumeration of $(2n+1,l)$-alternating sign trapezoids is given by
	\begin{equation*}
		2^{n+1} \left( \prod_{1 \le i \le j \le n} \frac{2(i+j)+l-3}{2i-1} \right)^2
	\end{equation*}
	and the $2$-enumeration of $(2n+2,l)$-alternating sign trapezoids is given by
	\begin{equation*}
	2^{n+1} \prod_{1 \le i \le j \le n} \frac{2(i+j)+l-3}{2i-1} \prod_{1 \le i \le j \le n+1} \frac{2(i+j)+l-3}{2i-1} .
	\end{equation*}
\end{theorem}

\begin{proof}
Let $l \ge 1$. To prove the conjecture, we provide a product formula for $\mathcal{Z}_{\AST}(n,l;2,1,1,1)$. To this end, we set
\begin{equation*}
\label{eq:funkdet}
f(l;i,j) \coloneqq \sum_{k = 0}^{i} \sum_{m=0}^{j} 2^{k-m} \binom{j}{m} \left( \binom{k+l-3}{k-m} + \frac{1}{2} \binom{k+l-3}{k-m-1} \right).
\end{equation*}
Then
\begin{equation*}
\mathcal{Z}_{\AST}(n,l;2,1,1,1) = \det_{0 \le i,j \le n-1} \left( f(l;i,j) + \delta_{i,j} \right).
\end{equation*}
First, we simplify the expression for $f(l;i,j)$. We split the sum up, shift the index in the second summand and make use of $ \binom{j}{m} + \binom{j}{m-1} = \binom{j+1}{m}$ in order to see that
\begin{equation*}
\sum_{m=0}^{j} 2^{k-m} \binom{j}{m} \left( \binom{k+l-3}{k-m} + \frac{1}{2} \binom{k+l-3}{k-m-1} \right)
= \sum_{m=0}^{j+1}  2^{k-m} \binom{j+1}{m} \binom{k+l-3}{k-m},
\end{equation*}
which can be restated as an ordinary hypergeometric function\footnote{The \emph{ordinary hypergeometric function} is defined as $\pFq{2}{1}{a,b}{c}{z} \coloneqq \sum_{n=0}^{\infty} \frac{\left(a\right)_n \left(b\right)_n}{\left(c\right)_n} \frac{z^n}{n!}$, where $\left(a\right)_n \coloneqq a \left(a+1\right) \cdots \left(a+n-1\right)$ if $n \ge 1$ and $\left(a\right)_0 \coloneqq 1$ denote the \emph{Pochhammer symbol}.} as follows by means of an instance of the classical Pfaff transformation \cite[(2.2.6)]{AAR99}:
\begin{equation*}
2^k \binom{k+l-3}{k} \pFq{2}{1}{-k,-j-1}{l-2}{\frac{1}{2}} = \binom{k+l-3}{k} \pFq{2}{1}{-k,j+l-1}{l-2}{-1}.
\end{equation*}
Finally, by resorting the order of summation, we further deduce that $f(l;i,j)$ can be expressed as
\begin{multline*}
\sum_{k = 0}^{i} \binom{k+l-3}{k} \pFq{2}{1}{-k,j+l-1}{l-2}{-1}\\
= \sum_{k = 0}^{i} \binom{k+l-3}{k} \sum_{m=0}^{k} \binom{k}{m} \frac{\left(j+l-1\right)_m}{\left(l-2\right)_m} = \sum_{m=0}^{i}  \frac{\left(j+l-1\right)_m}{\left(l-2\right)_m} \sum_{k = m}^{i} \binom{k+l-3}{k} \binom{k}{m}\\
= \sum_{m=0}^{i}  \frac{\left(j+l-1\right)_m}{\left(l-2\right)_m} \binom{m+l-3}{m} \sum_{k = 0}^{i-m} \binom{k+m+l-3}{k} = \sum_{m=0}^{i}  \frac{\left(j+l-1\right)_m}{m!} \binom{i+l-2}{m+l-2},
\end{multline*}
which is equal to
\begin{equation}
	\label{eq:HypDet}
\binom{i+l-2}{l-2} \pFq{2}{1}{-i,j+l-1}{l-1}{-1}.
\end{equation}
However, if we set $a=l-2$, then
\begin{equation*}
g(a;i,j) \coloneqq \sum_{m = 0}^{i} \binom{m+j+a}{m}\binom{i+a}{m+a}
\end{equation*}
equals \eqref{eq:HypDet}, too. Hence, $g(l-2;i,j)=f(l;i,j)$. Andrews \cite[Theorem~1 \& Lemma~2]{And87} evaluated the determinant $D_n(a) \coloneqq \det_{0 \le i,j \le n-1} \left( g(a;i,j) + \delta_{i,j} \right)$ for $a \ge 0$ and proved that
\begin{align}
\frac{D_{2n}(a)}{D_{2n-1}(a)} &= 2^n \prod_{k=1}^{n} \frac{2\left(n+k\right)+a-1}{n+k},\label{eq:AndrewsD1}\\
\frac{D_{2n+1}(a)}{D_{2n}(a)} &= 2^{n+1} \prod_{k=1}^{n} \frac{2\left(n+k\right)+a-1}{n+k}.\label{eq:AndrewsD2}
\end{align}
Since $D_1(a)=2$, \eqref{eq:AndrewsD1} and \eqref{eq:AndrewsD2} imply Theorem~\ref{thm:2Enum}.\end{proof}

The determinant $D_n(a)$ originally emerged in \cite{MRR87} as a generating function of a differently weighted enumeration of column strict shifted plane partitions of class $2a$ with at most $n$ parts in the first row. 

Note that $\mathcal{Z}_{\AST}(n,1;x,1,1,1)$ is not equivalent to the $x$-enumeration of quasi alternating sign triangles since we altered the weight in the case $l=1$. However, computer experiments suggest that neither the $2$-enumeration $2$, $5$, $22$, $188$, $3152$, $104704$, $6905856,\dots$ nor the $3$-enumeration $2$, $5$, $24$, $252$, $5832$, $301077$, $34720812,\dots$ of quasi alternating sign triangles yields round numbers. Computer experiments suggest that the $3$-enumeration of $(n,l)$-alternating sign trapezoids does not provide round numbers either unless $l=3$. Furthermore, note that for $l=2$, the $2$-enumeration coincides with the total dimension of the homology of free $2$-step nilpotent Lie algebras \cite[Theorem~1.1]{GKT02} as well as with the number of rhombus tilings of half-hexagons with glued sides \cite[(3.5)]{DFZJZ05}. 

\section*{Acknowledgement}
I thank the anonymous referee for the careful reading of my manuscript and for the various helpful comments and suggestions.

\bibliography{ASM}
\bibliographystyle{halpha}

\textsc{Universit{\"a}t Wien, Fakult{\"a}t f{\"u}r Mathematik, Oskar-Morgenstern-Platz 1, 1090 Wien, Austria}

\textit{E-mail address:} \href{mailto:hans.hoengesberg@univie.ac.at}{hans.hoengesberg@univie.ac.at}

\end{document}